\documentclass[onecolumn,conference,letterpaper]{IEEEtran}

\usepackage[utf8]{inputenc} 
\usepackage[T1]{fontenc}
\usepackage{url}
\usepackage{ifthen}
\usepackage{cite}
\usepackage[cmex10]{amsmath} 

\usepackage{color}
\usepackage{graphicx}
\usepackage{amssymb, amsmath, amsthm, mathtools}
\usepackage{bm}
\usepackage{comment}
\usepackage[colorlinks=true]{hyperref}

\usepackage{algorithm,algorithmic}
\newenvironment{varalgorithm}[1]
{\algorithm\renewcommand{\thealgorithm}{#1}}
{\endalgorithm}

\hypersetup{
     colorlinks = true,
     linkcolor = [rgb]{0,0,1},
     anchorcolor = [rgb]{0,0,1},
     citecolor = [rgb]{0.9,0.5,0},
     filecolor = [rgb]{0,0,1},
     pagecolor = [rgb]{1,1,0},
     urlcolor = [rgb]{0.9,0.5,0},
     bookmarks,
        bookmarksopen = true,
        bookmarksnumbered = true,
        breaklinks = true,
        linktocpage,
        pagebackref,
        colorlinks = true,
        linkcolor = [rgb]{0.2,0.6,0.2},
        urlcolor  = [rgb]{0,0,1},
        citecolor = [rgb]{0.9,0.5,0},
        anchorcolor = [rgb]{0.2,0.6,0.2},
        hyperindex = true,
        hyperfigures
}

\newtheorem{thm}{\bf{Theorem}}[section]
\newtheorem{proposition}{\bf{Proposition}}[section]
\newtheorem{remark}{\bf{Remark}}[section]
\newcommand{\nn}{\nonumber}
\newcommand{\avertot}{average~ total}

\def\lb{\mathop{\mathrm{LB}}}   

\begin{document}
\title{Quadratic Signaling Games with Channel Combining Ratio} 

 \author{%
   \IEEEauthorblockN{Serkan~Sar{\i}ta\c{s}, Photios~A.~Stavrou, Ragnar~Thobaben and Mikael~Skoglund}
   \IEEEauthorblockA{Division of Information Science and Engineering\\
   KTH Royal Institute of Technology\\
   SE-10044, Stockholm, Sweden\\
   Email: \{saritas, fstavrou, ragnart, skoglund\}@kth.se}
 }

\maketitle

\begin{abstract}
In this study, Nash and Stackelberg equilibria of single-stage and multi-stage quadratic signaling games between an encoder and a decoder are investigated. In the considered setup, the objective functions of the encoder and the decoder are misaligned, there is a noisy channel between the encoder and the decoder, the encoder has a soft power constraint, and the decoder has also noisy observation of the source to be estimated. We show that there exist only linear encoding and decoding strategies at the Stackelberg equilibrium, and derive the equilibrium strategies and costs. Regarding the Nash equilibrium, we explicitly characterize affine equilibria for the single-stage setup and show that the optimal encoder (resp. decoder) is affine for an affine decoder (resp. encoder) for the multi-stage setup. For the decoder side, between the information coming from the encoder and noisy observation of the source, our results describe what should be the combining ratio of these two channels. Regarding the encoder, we derive the conditions under which it is meaningful to transmit a message.
\end{abstract}


\section{Introduction}

Decision making has a wide-range of applications, from  engineering areas (e.g. information and communication theories, control theory, machine learning etc.) to social sciences (e.g. economics, management etc.) to interdisciplinary sciences (e.g. cognitive science). Every decision mechanism requires some prior input/data or observation for the decision maker (DM) so that an optimal decision can be made. The question may arise if, for instance, there are multiple observations corresponding to the same data, are all the inputs reliable, which observation is the most usable one etc. In this paper, we search for an answer to these questions with two observation channels under a game theoretic framework.

Consider a scenario with two DMs, an encoder and a decoder. The encoder has access to the data and transmits a message to the decoder over a noisy channel. Besides information coming from the encoder, the decoder has also access to a noisy observation of the original data. Based on these two observations/inputs, the decoder takes its optimal action. Here, the encoder and the decoder are assumed to have misaligned objective functions, which makes the setting a game theoretic setup. In the following, we make further explanations and comments:
\begin{itemize}
	\item Our setup can be considered as a signaling game: A privately informed sender (i.e., encoder) observes the private data and chooses a signal that is observed by the (uninformed) receiver (i.e., decoder). Upon receiving the message from the encoder, the decoder picks an action, which determines the costs\footnote{If the transmitted signal does not affect the costs, the game is called as \textit{cheap talk}.} of the encoder and the decoder. 
	\item From the decoder's perspective, there are two information sources: a noisy observation of the encoder's message and of the original data. The considered question is then, which conditions dictate channels combining\footnote{The decoder utilizes the convex combination of the channels, and uses restricted gain coefficients for the utilization of channels (e.g. due to power constraint), thus our setup is not equivalent to the case of parallel Gaussian channels, and it may not achieve maximum-ratio combining (see Remark~\ref{rem:commentDecoder}).} and what should be their respective ratio of utilization.
	\item The setup can also be considered as a point-to-point communication setup with (specific type of) side information\footnote{Since the decoder cannot adjust the gains of the main and side channels separately, our setup is not completely equivalent to the point-to-point communication setup with side information at the decoder.} at the decoder.
\end{itemize} 

\subsection{Motivational Example}

When satellite navigation such as GPS is inadequate due to various reasons (e.g., signal loses significant power indoors, multiple reflections may cause multi-path propagation or acquiring a satellite fix may take too long), additional information such as Wi-Fi positioning systems and indoor positioning systems can be utilized. As a solution to this problem, i.e., in order to make positioning signals ubiquitous, integration between satellite navigation and indoor positioning can be made. Accordingly, our setup can model such a scenario: Actual location is to be estimated by the user, and satellite navigation, which contains the location information, can be considered as the original data. An analogous of the encoder is the other positioning systems, which transmits location related information to the user. Even though there is a direct noisy channel between the original data (actual location) and the user due to satellite navigation, more precise location estimate can be achieved by utilizing additional information coming from the other positioning systems. 

\subsection{Related Literature}

The studies on cheap talk and signaling games are initiated by Crawford and Sobel in \cite{SignalingGames}, who showed that under some technical conditions on the objective functions of the players, the cheap talk problem only admits quantized Nash equilibrium strategies. Signaling games have many applications in networked systems \cite{misBehavingAgents,csLloydMax}, recommendation systems \cite{miklos2013value,recommSystemGame}, and economics \cite{signalSurvey,Sobel2009}. 

Starting with a seminal work \cite{bayesianPersuasion}, there are many studies that consider the Stackelberg equilibrium of signaling games \cite{tacWorkSerkan,CedricWork,akyolITapproachGame,omerHierarchial,dynamicGameSerkan,strategicCommSideInfo,serkanACC2020}. Many of these works assume that the non-alignment between the objective functions of the encoder and the decoder is a function of a Gaussian random variable (RV) correlated with the Gaussian source and secret to the decoder (unlike the original case where it is fixed and commonly known by the encoder and the decoder \cite{SignalingGames}, which is also studied in \cite{tacWorkSerkan,dynamicGameSerkan,serkanACC2020} and in this paper), the Stackelberg equilibrium under quadratic costs is investigated in \cite{CedricWork, akyolITapproachGame,omerHierarchial}. We refer \cite{Sobel2009,tacWorkSerkan,dynamicGameSerkan} for more discussion on the literature and some extensions (including Nash equilibrium analyses and multi-stage extensions) on cheap talk and signaling games. 

An information theoretic formulation of the Bayesian persuasion problem \cite{bayesianPersuasion} is studied in \cite{akyolITapproachGame} for general (not necessarily Gaussian) sources, including the case with side information at the decoder, and recently also in \cite{strategicCommSideInfo} with finite state and action spaces by assuming a decoder side information, respectively. In \cite{noisyWynerZiv}, lossy source coding with side information at the decoder only, known as the Wyner-Ziv coding, is studied in which the source is observed via a memoryless noisy channel. 

Similar to our setup, the Bayesian Nash equilibrium of a finite alphabet semantic communication game is investigated in \cite{semCommGame}. Besides the encoder/decoder pair acting as a team, there is also an (helpful or adversarial) agent who is able to modify the  channel transition probability of the side information received by the decoder. In \cite{fadingSideInfo}, a similar setup is considered in which the decoder, besides receiving the message from the encoder over a noiseless channel, also observes side information consisting of the original source subject to slow fading and noise. The source coding analysis in \cite{fadingSideInfo} is extended to the joint source/channel coding analysis by assuming a noisy but static channel between the encoder and the decoder \cite{jointSourceChannelCodingSide}.

\subsection{Contributions}

The main contributions of this paper can be summarized
as follows:
\begin{itemize}
	\item[(i)] A signaling game between an encoder and a decoder with quadratic objective functions is modeled with channel combining and utilization at the decoder side. 
	\item[(ii)] Nash and Stackelberg equilibria of the single-stage and multi-stage setups are investigated, and the equilibrium strategies and costs are characterized.
	\item[(iii)] The optimality of linear strategies is proved for the single-stage Stackelberg equilibrium (Theorem~\ref{thm:linearEq}).
	\item[(iv)] For the Stackelberg equilibrium of the multi-stage setup, it is proved that the linear strategies are optimal for both the encoder and the decoder (Theorem~\ref{theorem:enc_sol_soft}), and an algorithm is provided to find the equililbrium (Algorithm~\ref{algo2}).
	\item[(v)] For the Nash equilibrium of the single-stage and multi-stage setups, it is proved that the optimal encoder (decoder) is affine for an affine decoder (encoder) (Theorem~\ref{thm:affineNash}).
\end{itemize}


The remainder of the paper is organized as follows. We present the system model and problem formulation in Section II. Stackelberg equilibria with single-stage and multi-stage are investigated in Section III and Section IV, respectively. In Section V, we analyze Nash equilibria. Section VI concludes the paper and discusses future research directions.

\noindent\textbf{Notations:} $\mathcal{N}(\mu,\sigma^2)$ denotes a scalar Gaussian distribution with mean $\mu$ and variance $\sigma^2$, and we denote random variables by bold lower case letters, e.g., ${\bf x}$.

\section{System Model and Problem Formulation}

\subsection{System Model}

For the purpose of illustration, the considered system model is depicted in Fig.~\ref{figure:combinedSide}. An informed player (encoder) observes the realization of the scalar Gaussian RV ${\bf x}\sim\mathcal{N}(0,\sigma_{\bf x}^2)$ and transmits a message ${\bf m}$ to the uninformed player (decoder) through the additive white Gaussian noise channel (AWGN). The noise ${\bf v}$ is modeled as ${\bf v}\sim\mathcal{N}(0,\sigma_{\bf v}^2)$, and the output ${\bf y}$ of the channel is ${\bf y}={\bf m}+{\bf v}$. Besides the noisy message from the encoder, the decoder has also an access to the source over an AWGN channel, i.e., the decoder can also observe ${\bf z}={\bf x}+{\bf w}$ with ${\bf w}\sim\mathcal{N}(0,\sigma_{\bf w}^2)$. The decoder can choose to observe either of the channels or combination of them (e.g., by using a time-sharing approach). In particular, letting $\alpha\in[0,1]$, the combining\&utilization ratio of the channel from the encoder is $\alpha$ whereas of the channel from the source is $1-\alpha$, i.e.\footnote{Combining\&utilization of channels can be interpreted as channel gains of $\alpha$ and $1-\alpha$. Modifying the channels' gains as $\alpha_1\in\mathbb{R}$ and $\alpha_2\in\mathbb{R}$ results in infinitely many decoder strategies and maximum-ratio combining (see Remark~\ref{rem:commentDecoder}).}, ${\bf r}=\alpha {\bf y} + (1-\alpha){\bf z}$. The decoder, upon observing its input ${\bf r}$, generates an estimate ${\bf\hat{x}}$ of the original source ${\bf x}$.  

\begin{figure} [ht]
	\begin{center}
		\includegraphics[width=0.6\columnwidth]{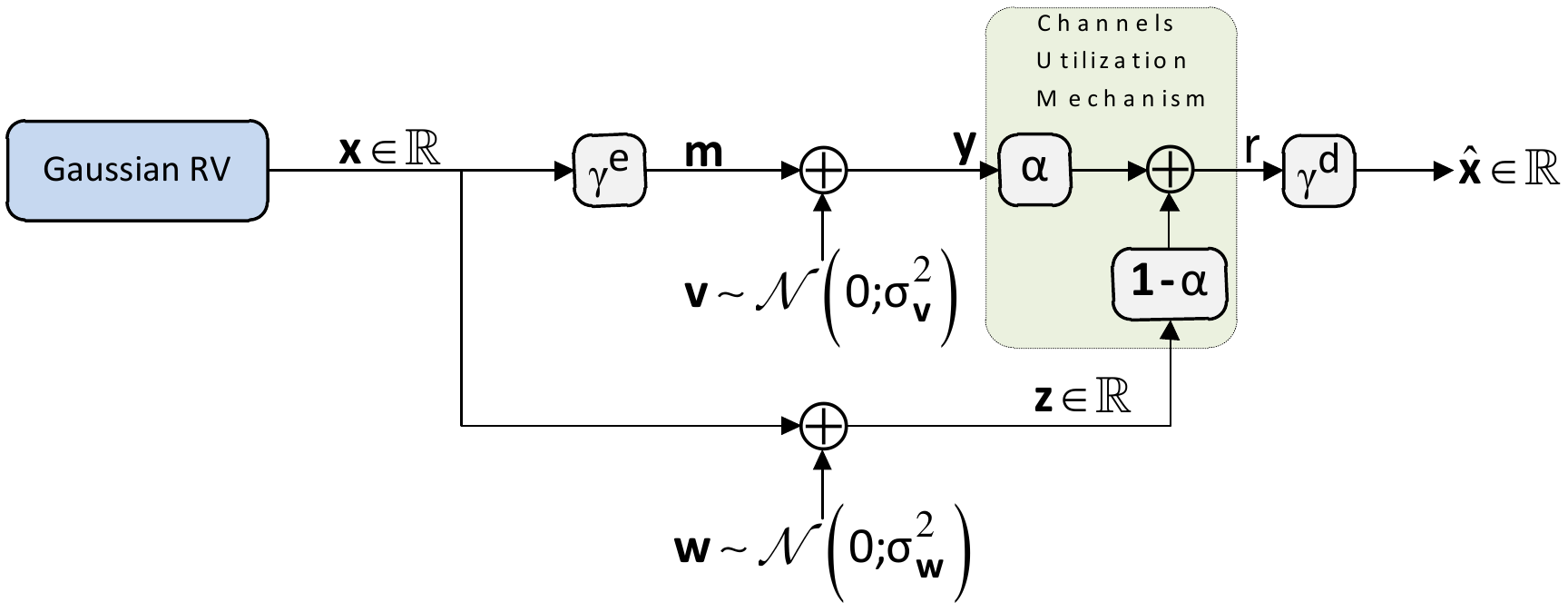}
	\end{center}
	\caption{Single-stage system model.}
	\label{figure:combinedSide}
\end{figure}

\subsection{Preliminaries}

For the source realization $x$ and the decoder estimate $\hat{x}$, let $c^e(x,\hat{x})$ and $c^d(x,\hat{x})$ denote the corresponding cost functions of the encoder and the decoder, respectively. Then, for the given encoder strategy ${\bf m}=\gamma^e({\bf x})$ and the decoder strategy $\hat{{\bf x}}=\gamma^d({\bf r})$, the expected encoder and the decoder costs are $J^e\left(\gamma^e,\gamma^d\right) = {\bf E}\left[c^e({\bf x},\hat{{\bf x}})\right]$ and $J^d\left(\gamma^e,\gamma^d\right) = {\bf E}\left[c^d({\bf x},\hat{{\bf x}})\right]$, respectively. Since the costs are not (essentially) equivalent/aligned, the problem is studied under a game theoretic framework, and two equilibrium types are investigated: Stackelberg and Nash equilibria.

In the Stackelberg (leader-follower) game, the leader (encoder) commits to a particular policy and announces it to the follower (decoder). Upon observing the encoder's committed strategy, the decoder takes its optimal action. More precisely, a pair of strategies $(\gamma^{e,*}, \gamma^{d,*})$ is said to be a {\it Stackelberg equilibrium} \cite{basols99} if
\begin{align}
\begin{split}
&J^e(\gamma^{e,*}, \gamma^{d,*}(\gamma^{e,*})) \leq J^e(\gamma^e, \gamma^{d,*}(\gamma^e)) \quad \forall \gamma^e \in \Gamma^e \,,\\
&\hspace{-0.5cm} \text{where } \gamma^{d,*}(\gamma^e) \text{ satisfies} \\
&J^d(\gamma^{e}, \gamma^{d,*}(\gamma^{e})) \leq J^d(\gamma^{e}, \gamma^d(\gamma^{e})) \quad \forall \gamma^d \in \Gamma^d  \,.
\label{eq:stackelbergEquilibrium}
\end{split}
\end{align}
Note that the follower (decoder) takes its action after observing the strategy $\gamma^{e}$ of the leader (encoder), the strategy $\gamma^d(\gamma^{e})$ of the decoder is a function of $\gamma^{e}$.

In the Nash (simultaneous-move) game, the encoder and the decoder announce their strategies at the same time. More precisely, a pair of policies $(\gamma^{e,*}, \gamma^{d,*})$ is said to be a {\it Nash equilibrium} \cite{basols99} if
\begin{align}
\begin{split}
J^e(\gamma^{e,*}, \gamma^{d,*}) &\leq J^e(\gamma^{e}, \gamma^{d,*}) \quad \forall \gamma^e \in \Gamma^e \,,\\
J^d(\gamma^{e,*}, \gamma^{d,*}) &\leq J^d(\gamma^{e,*}, \gamma^{d}) \quad \forall \gamma^d \in \Gamma^d \,.
\label{eq:nashEquilibrium}
\end{split}
\end{align}
As observed in \eqref{eq:nashEquilibrium}, none of the players prefers to change their optimal strategies at the equilibrium, i.e., there is no unilateral profitable deviation from any of the players. 

\subsection{Problem Formulation}

We consider quadratic cost functions with a soft power constraint at the encoder side. In particular, $c^e(x,\hat{x})=(x-\hat{x}-b)^2+\theta (\gamma^e(x))^2$ and $c^d(x,\hat{x})=(x-\hat{x})^2$, where $b$ denotes the bias term commonly known by the encoder and the decoder, i.e., the misalignment between the encoder and the decoder costs, and $\theta$ is a coefficient responsible for the soft power constraint. Note that the costs simply reduce to those for a minimum mean-square estimation (MMSE) problem when $b=0$. Further note that the case with $\theta=0$ corresponds to the setup with no power constraint at the encoder.

The encoder aims to minimize $J^e\left(\gamma^e,\gamma^d\right) = {\bf{E}}\left[c^e({\bf x},\hat{{\bf x}})\right]$ by selecting an optimal encoding strategy $\gamma^e({\bf x})$ whereas the decoder's goal is to minimize $J^d\left(\gamma^e,\gamma^d\right) = {\bf{E}}\left[c^d({\bf x},\hat{{\bf x}})\right]$ by choosing an optimal decoding strategy $\gamma^d({\bf r})$ and the channel combining parameter $\alpha$.

\section{Single-Stage Stackelberg Equilibrium}\label{sec:staticStackelberg}

In this section, we analyze the Stackelberg equilibrium of the game between the encoder (leader) and the decoder (follower). First, we show that the lowest estimation error is achieved when the encoder and the decoder jointly use linear strategies. Then we characterize the (existence of) equilibria with respect to the soft power coefficient $\theta$.

\begin{thm}\label{thm:affineDecoder}
	Let the encoder use a linear strategy such that ${\bf m}=\gamma^e({\bf x})=A{\bf x}$. Then, the optimal decoder selects the channel combining parameter $\alpha$ and the linear strategy $\hat{{\bf x}}=\gamma^d({\bf r})=K{\bf r}$ correspondingly. The optimal decoder strategy $\alpha^*$ and $K^*$, and its corresponding cost $J^{d,*} = {\bf E}[({\bf x}-\hat{{\bf x}})^2]$ are characterized in Table~\ref{tab:optDecoderCombined}.
	\begin{table}[ht]
		\caption{Optimal decoder strategy for a linear encoder.}
		\label{tab:optDecoderCombined}
		\centering
		\begin{tabular}{|c|c|c|c|}
			\hline
			\text{Case} & $\alpha^*$ & $K^*$ & $J^{d,*}$  \\ \hline &&&\\[-1em]
			$A \geq 0$	& ${A\sigma_{\bf w}^2 \over A\sigma_{\bf w}^2+\sigma_{\bf v}^2}$ & ${A\sigma_{\bf x}^2\sigma_{\bf w}^2 + \sigma_{\bf x}^2\sigma_{\bf v}^2\over A^2\sigma_{\bf x}^2\sigma_{\bf w}^2+\sigma_{\bf x}^2\sigma_{\bf v}^2+\sigma_{\bf w}^2\sigma_{\bf v}^2}$ & ${\sigma_{\bf x}^2\sigma_{\bf w}^2\sigma_{\bf v}^2 \over (A^2\sigma_{\bf w}^2+\sigma_{\bf v}^2)\sigma_{\bf x}^2+\sigma_{\bf w}^2\sigma_{\bf v}^2}$  \\ \hline &&&\\[-1em]
			$-\sqrt{\sigma_{\bf v}^2\over\sigma_{\bf w}^2}\leq A\leq0$	& $0$ & ${\sigma_{\bf x}^2 \over \sigma_{\bf x}^2+\sigma_{\bf w}^2}$ & ${\sigma_{\bf x}^2\sigma_{\bf w}^2 \over \sigma_{\bf x}^2+\sigma_{\bf w}^2}$ \\ \hline &&&\\[-1em]
			$A\leq-\sqrt{\sigma_{\bf v}^2\over\sigma_{\bf w}^2}$	& $1$ & ${A\sigma_{\bf x}^2 \over A^2\sigma_{\bf x}^2+\sigma_{\bf v}^2}$ & ${\sigma_{\bf x}^2\sigma_{\bf v}^2 \over A^2\sigma_{\bf x}^2+\sigma_{\bf v}^2}$  \\ \hline
		\end{tabular}%
	\end{table}
\end{thm}
\IEEEproof
See Appendix~\ref{appendixaffineDecoder}.
\endIEEEproof

\begin{remark}\label{rem:commentDecoder}
	As it can be observed from Table~\ref{tab:optDecoderCombined}, the optimal decoder achieves maximum-ratio combining by randomizing the channels when $A>0$. However, when $A<0$, since the decoder's action space does not support maximum-ratio combining, the decoder always selects the better channel without randomization.
\end{remark}

\begin{thm}\label{thm:linearOpt}
	The lower bound on the estimation error $J^d=\bf{E}[({\bf x}-\hat{{\bf x}})^2]$ is $\frac{\sigma_{\bf x}^2}{{P\over\sigma_{\bf v}^2}+{\sigma_{\bf x}^2\over\sigma_{\bf w}^2}+1}$ where $P\triangleq\bf{E}[{\bf m}^2]$ is the power of the transmitted signal ${\bf m}=\gamma^e({\bf x})$ by the encoder, and this lower bound is achieved if and only if both the encoder and the decoder jointly use linear strategies.
\end{thm}
\IEEEproof
See Appendix~\ref{appendixlinearOpt}.
\endIEEEproof

Regarding the encoder cost, observe the following\footnote{Since we assume a fixed and public $b$ in contrast to a private and random $b$ which is correlated with the source as in \cite{CedricWork, akyolITapproachGame,omerHierarchial}, the results obtained in the former setup cannot be applied directly to the latter one; i.e., the Stackelberg equilibria of these two setups are different.}.
\begin{remark}\label{rem:encCost}
	Due to the Stackelberg assumption, since the encoder anticipates that the decoder will use $\hat{{\bf x}}=\gamma^{d,*}({\bf r})=\bf{E}[{\bf x}|{\bf r}]$, the bias $b$ can be decoupled from the encoder cost \cite{tacWorkSerkan,dynamicGameSerkan}. In particular,
	\begin{align*}
	J^e &= {\bf{E}} [({\bf x}-{\bf{E}}[{\bf x}|{\bf r}]-b)^2+\theta (\gamma^e({\bf x}))^2] \nn\\
	&= {\bf{E}} [({\bf x}-{\bf{E}}[{\bf x}|{\bf r}])^2+\theta (\gamma^e({\bf x}))^2]+b^2 \nn\\
	&= J^d+{\bf{E}} [\theta (\gamma^e({\bf x}))^2]+b^2 \,.
	\end{align*}
\end{remark}

After finding the optimal decoder cost, we can proceed to analyze the optimum encoder strategy and characterize the (existence of) equilibria.
\begin{thm}\label{thm:linearEq}
The only equilibrium (affine or not) in the Stackelberg setup is the linear equilibrium with $\gamma^e({\bf x})=A{\bf x}$ and $\gamma^d({\bf r})=K{\bf r}$ with $A\geq0$ and $K\geq0$. \newline
In particular, at the equilibrium, the encoder cost is
\[J^{e,*} = J^{d,*} + \theta A^2\sigma_{\bf x}^2 + b^2 \,,\]
where $A$ is decided according to the following decision rule
\begin{align*}
A^2=\begin{cases}
{\sqrt{\sigma_{\bf v}^2\over\theta{\sigma_{\bf x}^2}}-{\sigma_{\bf v}^2\over\sigma_{\bf x}^2}\left({{\sigma_{\bf x}^2\over\sigma_{\bf w}^2}+1}\right)},~&\mbox{if}\qquad~\mbox{$\theta<{{\sigma_{\bf x}^2} \over \sigma_{\bf v}^2\left({\sigma_{\bf x}^2\over\sigma_{\bf w}^2}+1\right)^2}$}\\
0,~&\mbox{if}\qquad~\mbox{$\theta\geq{{\sigma_{\bf x}^2} \over \sigma_{\bf v}^2\left({\sigma_{\bf x}^2\over\sigma_{\bf w}^2}+1\right)^2}$}\\
\end{cases}\,.
\end{align*}
Then, the corresponding $\alpha$, $K$, and $J^{d,*}$ can be derived from Table~\ref{tab:optDecoderCombined}.
\end{thm}
\IEEEproof
See Appendix~\ref{appendixlinearEq}.
\endIEEEproof

\section{Multi-Stage Stackelberg Equilibrium}

In this section, we consider the dynamic counterpart of Section \ref{sec:staticStackelberg}. We start by giving the problem statement illustrated in Fig. \ref{fig:dynamic_setup1}.

\begin{figure}[htp]
	\centering
	\includegraphics[width=0.6\columnwidth]{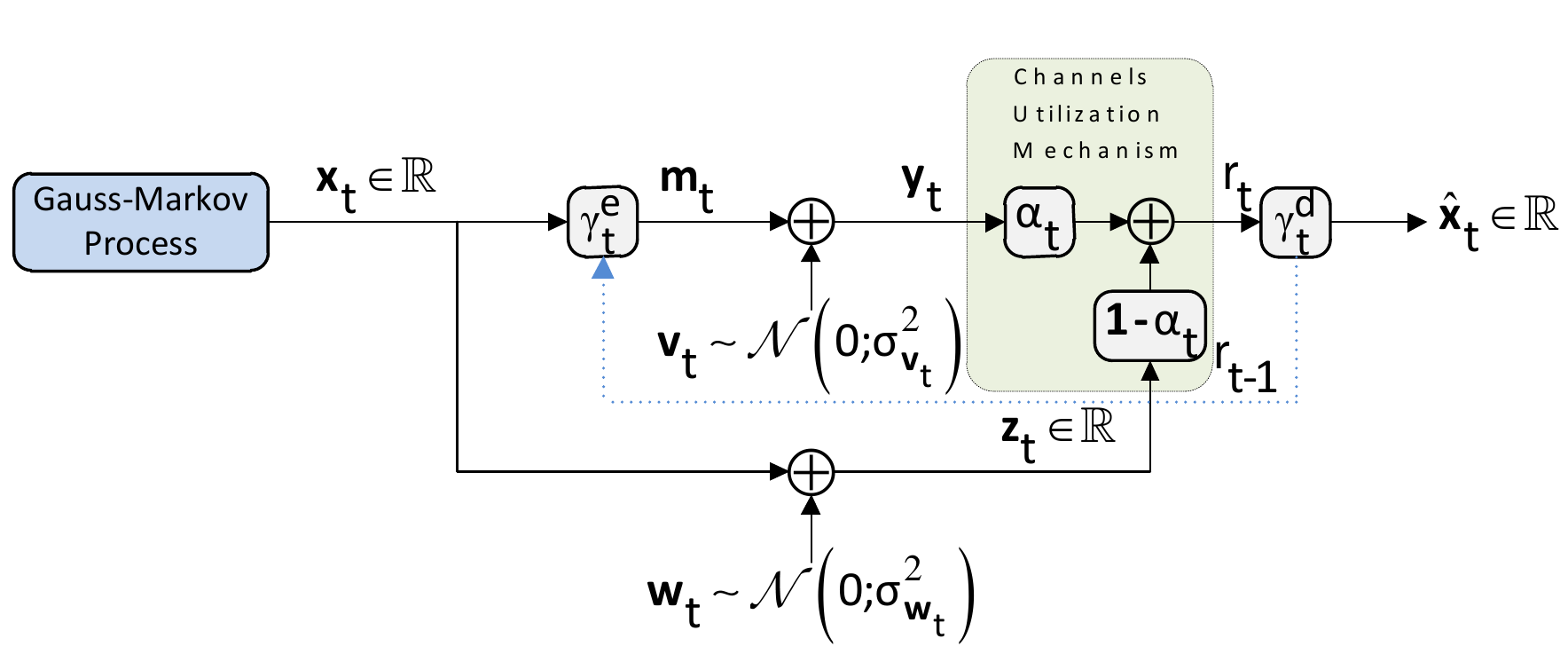}
	\caption{Multi-stage system model.} 
	\label{fig:dynamic_setup1}
\end{figure}

In Fig. \ref{fig:dynamic_setup1}, an input message is formed as a Gauss-Markov model described by the following recursion:
\begin{align}
{\bf x}_{t+1}=\beta_t{\bf x}_{t}+{\bf n}_t,~t\in\mathbb{N}_0^n,
\label{gauss-markov_model}
\end{align}
where $\{{\bm \beta}_t:~t\in\mathbb{N}_0^n\}$ is  a deterministic coefficient, the initial message ${\bf x}_0\sim{\cal N}(0;\sigma^2_{{\bf x}_0}),~\sigma^2_{{\bf x}_0}>0$, and $\{{\bf n}_t:~t\in\mathbb{N}_0^{n-1}\}$ is a mutually independent process independent of everything with ${\bf n}_t\sim{\cal N}(0;\sigma^2_{{\bf n}_t}),~ \sigma^2_{{\bf n}_t}>0$.
\par We assume a causal noisy observation of source before the decoder, modeled as a time-varying Gaussian process as follows:
\begin{align}
{\bf z}_t={\bf x}_t+{\bf w}_t,~t\in\mathbb{N}_0^n,\label{csi}
\end{align}
where $\{{\bf w}_t:~t\in\mathbb{N}_0^n\}$ is an independent noise process independent of everything with ${\bf w}_t\sim{\cal N}(0;\sigma^2_{{\bf w}_t}),~ \sigma^2_{{\bf w}_t}>0$.

At stage $t$, the encoder has access to ${\bf x}^t\triangleq\{{\bf x}_0,{\bf x}_1,\ldots,{\bf x}_t\}$ and ${\bf r}^{t-1}\triangleq\{{\bf r}_0,{\bf r}_1,\ldots,{\bf r}_{t-1}\}$ (a
noiseless feedback channel is assumed) whereas ${\bf r}^{t}\triangleq\{{\bf r}_0,{\bf r}_1,\ldots,{\bf r}_t\}$ is available to the decoder. Then, we can define the stage-wise costs of the players similar to the single-stage case, i.e., $c^e_t(x_t,\hat{x}_t)=(x_t-\hat{x}_t-b_t)^2+\theta_t (\gamma^e(x^t,r^{t-1}))^2$ and $c^d_t(x_t,\hat{x_t})=(x_t-\hat{x_t})^2$, where $b_t\in\mathbb{R}$ denotes the stage-wise bias term commonly known by the encoder and the decoder, and $\{\theta_t\in(0,\infty):~t\in\mathbb{N}_0^n\}$ are the stage-wise coefficients for the soft power constraints. Assuming myopic encoder and decoder strategies, the costs are defined as follows:
\begin{align}
J_{t}^d &=\min_{\gamma^d_t({\bf r}^{t})\,,\, {\alpha}_t\in[0,1]}{\bf E}[({\bf x}_t-\hat{\bf x}_{t})^2] \,, \nn\\
J_{t}^e & =\min_{\gamma^e_t({\bf x}^t,{\bf r}^{t-1})} J_{t}^d + {\bf E}[\theta_t(\gamma^e_t({\bf x}^t,{\bf r}^{t-1}))^2] +b_t^2 \nn\\
J_{\avertot}^d &=\frac{1}{n+1}\sum_{t=0}^nJ_t^d \,, \nn\\
J_{\avertot}^e &=\frac{1}{n+1}\sum_{t=0}^nJ_t^e \,.
\label{total_cost}
\end{align}

\begin{remark}
Note that in the sequel we see that although the costs of the encoder and decoder appear to form a nested optimization, they are not. In fact they can be decoupled to distinct time stages $J_0^e, J_1^e,\ldots, J_n^e$ for the encoder and $J_0^d, J_1^d,\ldots, J_n^d$ for the decoder, and solved independently moving forward in time. 
\end{remark}

Similar to the single-stage counterpart, first we show that the lowest estimation error is achieved when the encoder and the decoder jointly utilize linear strategies. In the following, we first find the optimal decoder for a linear memoryless encoder without any feedback. 

\begin{thm}\label{thm:dynamic_sol}
	Let the encoder use a linear memoryless strategy such that ${\bf m}_t=\gamma^e({\bf x}_t)=A_t{\bf x}_t,~t\in\mathbb{N}_0^n$. Then, the optimal decoder is obtained by a discrete time Kalman filter due to joint Gaussianity and admits closed form recursions. To present the recursions of the filter, we need to define the following conditional mean and conditional variances\footnote{Due to joint Gaussianity, the conditional variances are equivalent to the unconditional ones.}:
	\begin{align*}
	\hat{\bf x}_{t|t-1}&\triangleq{\bf E}[{\bf x}_t|{\bf r}^{t-1}], ~\Sigma_{t|t-1}\triangleq{\bf E}[({\bf x}_t-\hat{\bf x}_{t|t-1})^2|{\bf r}^{t-1}], \nn\\
\hat{\bf x}_{t|t}&\triangleq{\bf E}[{\bf x}_t|{\bf r}^{t}],~	\Sigma_{t|t}\triangleq{\bf E}[({\bf x}_t-\hat{\bf x}_{t|t})^2|{\bf r}^{t}].
	\end{align*}
	Then, $\{\hat{\bf x}_{t|t-1}, \Sigma_{t|t-1}, \hat{\bf x}_{t|t}, \Sigma_{t|t}:  t\in\mathbb{N}^n_0\}$ satisfy the following scalar-valued filtering recursions:
	\begin{align}
	\begin{split}
	\hat{\bf x}_{t|t-1}&=\beta_{t-1}\hat{\bf x}_{t-1|t-1},\\
	\Sigma_{t|t-1}&=\beta_{t-1}^2\Sigma_{t-1|t-1}+\sigma^2_{{\bf n}_{t-1}},~\Sigma_{0|-1}=\sigma^2_{{\bf x}_0},\\
	\hat{\bf x}_{t|t}&=\hat{\bf x}_{t|t-1}+K_t{\bf I}_t, \\
	{\bf I}_t&\triangleq {\bf r}_t-{\bf E}\left[{\bf r}_t|{\bf r}^{t-1}\right] =({\alpha}_tA_t+1-{\alpha}_t)\left( {\bf x}_t -\hat{\bf x}_{t|t-1} \right) \\
	&\qquad\qquad\qquad+(1-{\alpha}_t){\bf w}_t+{\alpha}_t{\bf v}_{t},~\text{(innovations)}\\
	\sigma^2_{{\bf I}_t}&=({\alpha}_tA_t+1-{ \alpha}_t)^2\Sigma_{t|t-1}+(1-{\alpha}_t)^2\sigma^2_{{\bf w}_t}+{\alpha}^2_t\sigma^2_{{\bf v}_{t}}\\
	K_t&=\frac{\Sigma_{t|t-1}({\alpha}_tA_t+1-{ \alpha}_t)}{\sigma^2_{{\bf I}_t}},~\text{(Kalman Gain)}\\
	\Sigma_{t|t}&=(1-K_t({\alpha}_tA_t+1-{ \alpha}_t))\Sigma_{t|t-1},
\end{split}\label{kf}
\end{align}
which corresponds to the optimal decoder's strategy
\begin{align}
{\gamma}_t^d({\bf r}^t)&=\hat{\bf x}_{t|t}=\hat{\bf x}_{t|t-1}+K_t{\bf I}_t \nn\\
&=\hat{\bf x}_{t|t-1}+K_t({\bf r}_t-({ \alpha}_tA_t+1-{\alpha}_t)\hat{\bf x}_{t|t-1}-{ \alpha}_tC_t).\label{optimal_dec}
\end{align}
Then, the optimal channel combining parameter $\alpha_t$, Kalman gain $K_t$ and stage-wise costs $J^{d,*}_t$ (or  $\Sigma^*_{t|t}$) are characterized in Table~\ref{tab:optDecoderMulti}. The optimal average cost is $ J_{\avertot}^{d,*}=\frac{1}{n+1}\sum_{t=0}^nJ_t^{d,*}$.
\begin{table}[ht]
	\caption{Stage-wise optimal decoder strategy for a linear encoder.}
	\label{tab:optDecoderMulti}
	\centering
	\begin{tabular}{|c|c|c|c|}
		\hline
		\text{Case} & $\alpha_t^*$ & $K_t^*$ & $J^{d,*}_t$  \\ \hline &&&\\[-1em]
		$A_t \geq 0$	& ${A_t\sigma_{{\bf w}_t}^2 \over A_t\sigma_{{\bf w}_t}^2+\sigma_{{\bf v}_t}^2}$ & ${A_t\Sigma^*_{t|t-1}\sigma_{{\bf w}_t}^2 + \Sigma^*_{t|t-1}\sigma_{{\bf v}_t}^2\over A_t^2\Sigma^*_{t|t-1}\sigma_{{\bf w}_t}^2+\Sigma^*_{t|t-1}\sigma_{{\bf v}_t}^2+\sigma_{{\bf w}_t}^2\sigma_{{\bf v}_t}^2}$ & ${\Sigma^*_{t|t-1}\sigma_{{\bf w}_t}^2\sigma_{{\bf v}_t}^2 \over (A_t^2\sigma_{{\bf w}_t}^2+\sigma_{{\bf v}_t}^2)\Sigma^*_{t|t-1}+\sigma_{{\bf w}_t}^2\sigma_{{\bf v}_t}^2}$  \\ \hline &&&\\[-1em]
		$-\sqrt{\sigma_{{\bf v}_t}^2\over\sigma_{{\bf w}_t}^2}\leq A_t\leq0$	& $0$ & ${\Sigma^*_{t|t-1} \over \Sigma^*_{t|t-1}+\sigma_{{\bf w}_t}^2}$ & ${\Sigma^*_{t|t-1}\sigma_{{\bf w}_t}^2 \over \Sigma^*_{t|t-1}+\sigma_{{\bf w}_t}^2}$ \\ \hline &&&\\[-1em]
		$A_t\leq-\sqrt{\sigma_{{\bf v}_t}^2\over\sigma_{{\bf w}_t}^2}$	& $1$ & ${A_t\Sigma^*_{t|t-1} \over A_t^2\Sigma^*_{t|t-1}+\sigma_{{\bf v}_t}^2}$ & ${\Sigma^*_{t|t-1}\sigma_{{\bf v}_t}^2 \over A_t^2\Sigma^*_{t|t-1}+\sigma_{{\bf v}_t}^2}$  \\ \hline
		\end{tabular}%
\end{table}
\end{thm}
\IEEEproof
See Appendix~\ref{appendixdynamic_sol}.
\endIEEEproof

Now assume that the encoder has a memory, and consider linear encoders with memory via noiseless feedback, i.e., 
\begin{align}
\gamma^e({\bf x}^t,{\bf r}^{t-1})=A_t({\bf x}_t-\hat{{\bf x}}_{t|t-1}) ,~t\in\mathbb{N}_0^n \,.
\label{linear_encoder_memory}
\end{align}
The following proposition shows that the memoryless encoder and the innovations encoder (i.e., encoder with a memory) generate the same innovations process. 
\begin{proposition}\label{prop:1} 
The innovations process at the decoder for the class of linear encoders in \eqref{linear_encoder_memory} generates the same information with the innovations process obtained for the class of linear memoryless encoders assumed in Theorem~\ref{thm:dynamic_sol} at each instant of time. Hence, the same values of ${\alpha_t^*}$, $K_t^*$ and $J_t^{d,*}$ given in Table II can be derived even if the encoder is an innovations encoder.
\end{proposition}
\IEEEproof
See Appendix~\ref{appendixprop1}.
\endIEEEproof

In what follows, we leverage Proposition IV.1 to prove a theorem that generalize Theorem \ref{thm:linearOpt} to the dynamic setup.
\begin{thm}\label{theorem:optimality_linear_policies_dynamic}
	The lower bound on the estimation error $J_{\avertot}^d=\frac{1}{n+1}\sum_{t=0}^n{\bf E}\left[({\bf x}_t-\hat{\bf x}_{t|t})^2\right]=\frac{1}{2}\sum_{t=0}^nJ^{d,*}_{t,\lb}$ with $\{J^{d,*}_{t,\lb}:~t\in\mathbb{N}_0^n\}$ computed forward in time as follows:
	\begin{align}
	J_{t,\lb}^{d,*}={\Sigma^*_{t|t-1}\sigma^2_{{\bf w}_t}\sigma^2_{{\bf v}_t}\over (\frac{P_t}{\Sigma^*_{t|t-1}}\sigma^2_{{\bf w}_t}+\sigma^2_{{\bf v}_t})\Sigma^*_{t|t-1}+\sigma^2_{{\bf w}_t}\sigma^2_{{\bf v}_t}},\label{optimal_estimation_error}
	\end{align}
with $\Sigma^*_{t|t-1}=\beta^2_{t-1}J^{d,*}_{t-1,\lb}+\sigma^2_{{\bf n}_{t-1}}$, $\Sigma_{0|-1}=\sigma^2_{{\bf x}_0}$,  and $P_t\triangleq{\bf E}\left[\widetilde{\bf m}^2\right]$, $\widetilde{\bf m}\triangleq{\bf m}_t-\hat{\bf m}_{t|t-1}$, $\hat{\bf m}_{t|t-1}\triangleq{\bf E}[{\bf m}_t|{\bf r}^{t-1}]$, is the power of the innovation $\widetilde{\bf m}$ of the transmitted signal of ${\gamma}_t^e({\bf x}^t,{\bf r}^{t-1})$ by an encoder with noiseless feedback at each instant of time. This lower bound is achieved if and only if both the encoder and the decoder jointly use linear strategies. 
\end{thm}
\IEEEproof
See Appendix~\ref{appendixoptimality_linear_policies_dynamic}.
\endIEEEproof

Observe the following regarding the optimization problem at the encoder:
\begin{align}
J_{\avertot}^{e}&=\min\limits_{\substack{{\bf m_t}=\gamma^e({\bf x}^t,{\bf r}^{t-1})}} \frac{1}{n+1} \sum_{t=0}^n J_{t}^{d,*} + {\bf E}[\theta_t({\bf m_t})^2] + b_t^2 \nn\\
&\geq \min\limits_{\substack{{\bf m_t}=\gamma^e({\bf x}^t,{\bf r}^{t-1})}} \frac{1}{n+1} \sum_{t=0}^n J_{t,\lb}^{d,*} + {\bf E}[\theta_t({\bf m_t})^2] + b_t^2 \nn\\
&\geq \min\limits_{\substack{{\bf m_t}=\gamma^e({\bf x}^t,{\bf r}^{t-1})}} \frac{1}{n+1} \sum_{t=0}^n {\Sigma^*_{t|t-1}\sigma^2_{{\bf w}_t}\sigma^2_{{\bf v}_t}\over (\frac{P_t}{\Sigma^*_{t|t-1}}\sigma^2_{{\bf w}_t}+\sigma^2_{{\bf v}_t})\Sigma^*_{t|t-1}+\sigma^2_{{\bf w}_t}\sigma^2_{{\bf v}_t}} + {\bf E}[\theta_t({\bf m}_t)^2] +b_t^2 \,.
\end{align}
Here, due to Theorem~\ref{theorem:optimality_linear_policies_dynamic}, the lower bound is achievable for a linear encoder, i.e., ${\bf m_t}=\gamma^e({\bf x}^t,{\bf r}^{t-1})=A_t({\bf x}_t-\hat{{\bf x}}_{t|t-1})$, which implies $P_t={\bf E}\left[\widetilde{\bf m}^2\right]={\bf E}\left[({\bf m}_t-\hat{\bf m}_{t|t-1})^2\right]=A_t^2\Sigma_{t|t-1}$. Then the optimization problem at the encoder becomes
\begin{align}
&J_{\avertot}^{e}=\min\limits_{\substack{A_t\geq{0},~t\in\mathbb{N}_0^n}} \frac{1}{n+1} \sum_{t=0}^n\left[{\Sigma^*_{t|t-1}\sigma^2_{{\bf w}_t}\sigma^2_{{\bf v}_t}\over (A_t^2\sigma^2_{{\bf w}_t}+\sigma^2_{{\bf v}_t})\Sigma^*_{t|t-1}+\sigma^2_{{\bf w}_t}\sigma^2_{{\bf v}_t}}+\theta_t\left(A_t^2\Sigma_{t|t-1}\right)+b_t^2\right],\label{enc_optimization_soft}
\end{align}
where $\Sigma^*_{0|-1}=\sigma^2_{{\bf x}_0}$. The solution is obtained recursively, forward in time in the next theorem.
\begin{thm}\label{theorem:enc_sol_soft}(Recursive solution forward in time of \eqref{enc_optimization_soft}) For given $\{\theta_t\in(0,\infty):~t\in\mathbb{N}_0^n\}$, $\Sigma^*_{0|-1}=\sigma^2_{{\bf x}_0}$, the solution of \eqref{enc_optimization_soft} is as follows:
	\begin{align}
	J_{\avertot}^{e,*}=\frac{1}{n+1}\sum_{t=0}^nJ_t^{e,*},~\label{enc_optimization_sol_soft}
	\end{align}
	where $\{J_t^{e,*}:~t=0,1,\ldots,n\}$ are computed forward in time by
	\begin{align}
	J_t^{e,*}&=\left[J_t^{d,*}+\theta_tA_t^{2,*}\Sigma^*_{t|t-1}+b_t^2\right],~\Sigma^*_{0|-1}=\sigma^2_{{\bf x}_0},\label{individual_cost_soft}\\
	J_t^{d,*}&={\Sigma^*_{t|t-1}\sigma^2_{{\bf w}_t}\sigma^2_{{\bf v}_t}\over (A_t^{2,*}\sigma^2_{{\bf w}_t}+\sigma^2_{{\bf v}_t})\Sigma^*_{t|t-1}+\sigma^2_{{\bf w}_t}\sigma^2_{{\bf v}_t}},\label{individual_prime_cost}
	\end{align}
	where $A_t^{2,*}$ are computed according to the following decision rule
	\begin{align}
	A_t^{2,*}=\begin{cases}
	&\frac{\sigma^2_{{\bf v}_t}}{\sqrt{\theta_t\Sigma^*_{t|t-1}\sigma^2_{{\bf v}_t}}}-\frac{\sigma^2_{{\bf v}_t}}{\Sigma^*_{t|t-1}}\left(\frac{\Sigma^*_{t|t-1}}{\sigma^2_{{\bf w}_t}}+1\right),~\qquad\\
	&\qquad\qquad\qquad\mbox{if}\qquad~\mbox{$\theta_t<\frac{\Sigma^*_{t|t-1}}{\sigma^2_{{\bf v}_t}\left(\frac{\Sigma^*_{t|t-1}}{\sigma^2_{{\bf w}_t}}+1\right)^2}$}\\
	&0,~\qquad\\
	&\qquad\qquad\qquad\mbox{if}\qquad~\mbox{$\theta_t\geq\frac{\Sigma^*_{t|t-1}}{\sigma^2_{{\bf v}_t}\left(\frac{\Sigma^*_{t|t-1}}{\sigma^2_{{\bf w}_t}}+1\right)^2}$}\\
	\end{cases}.\label{cases_optimal_A_soft}
	\end{align}
	and $\{\Sigma_{t|t-1}^*:~t=1,\ldots,n\}$ are computed forward in time by
	\begin{align}
	\Sigma_{t|t-1}^*=\beta_{t-1}^2J_t^{d,*}+\sigma^2_{n-1}.\label{covariance_cost_soft}
	\end{align}
\end{thm}
\IEEEproof
See Appendix~\ref{appendixenc_sol_soft}.
\endIEEEproof

In Algorithm~\ref{algo2}, we summarize the previous results by providing an iterative scheme to compute the multi-stage Stackelberg equilibrium. 
\begin{varalgorithm}{1}
	\caption{Multi-stage Stackelberg equilibrium}
	\begin{algorithmic}
		\STATE {\textbf{Initialize:} Set $\Sigma_{0|-1}^*=\sigma^2_{{\bf x}_0}$, choose $\{(\beta_{t}, \sigma^2_{{\bf n}_t}):~t\in\mathbb{N}_0^{n-1}\}$ of \eqref{gauss-markov_model}; choose $\{(\sigma^2_{{\bf w}_t},\sigma^2_{{\bf v}_t}):~t\in\mathbb{N}_0^n\}$; choose $\{\theta_t\in(0,\infty):~t\in\mathbb{N}_0^n\}$.} 
		\FOR {$t=0:n$}
		\IF {$t>0$}
		\STATE {Compute $\Sigma^*_{t|t-1}$ according to \eqref{covariance_cost_soft}.}
		\ENDIF
		\STATE {Compute $A_{t}^{2,*}$ according to \eqref{cases_optimal_A_soft}.}
		\STATE{Compute $J_t^{d,*}$ according to \eqref{individual_prime_cost}.}
	   \STATE{Compute $J_t^{e,*}$ according to \eqref{individual_cost_soft}.}
		\ENDFOR
		\STATE {Compute $J_{\avertot}^{e,*}$ according to \eqref{enc_optimization_sol_soft}.}
		\STATE {Compute $J_{\avertot}^{d,*}=\frac{1}{n+1}\sum_{t=0}^nJ^{d,*}_t$.}
	\end{algorithmic}
	\label{algo2}
\end{varalgorithm}

\section{Nash Equilibrium}

In this section, we analyze the Nash equilibrium of the game between the encoder and the decoder. We consider only affine equilibria.

\begin{thm}\label{thm:affineNash}
	Consider a single-stage scenario.
	\begin{itemize}
		\item[(i)] If the encoder is affine, the optimal decoder is also affine. 
		\item[(ii)] If the decoder is affine, the optimal encoder is also affine. 
		\item[(iii)] For $\theta<{{\sigma_{\bf x}^2\over\sigma_{\bf v}^2} \over \left({\sigma_{\bf x}^2\over\sigma_{\bf w}^2}+1\right)^2}$, there are two affine Nash equilibria. In particular, letting $A^*\triangleq \sqrt{\sqrt{{1\over\theta}{\sigma_{\bf v}^2\over\sigma_{\bf x}^2}} - {\sigma_{\bf v}^2\over\sigma_{\bf w}^2}-{\sigma_{\bf v}^2\over\sigma_{\bf x}^2}}$, two sets of $\gamma^e({\bf x})=A{\bf x}+C$, $\gamma^d({\bf r})=K{\bf r}+L$ and the channel combining parameter $\alpha$ are characterized as 
	\begin{table}[ht]
	\centering
	\begin{tabular}{|c|c|c|c|c|}
		\hline
		$A$ & $C$ & $K$ & $L$ & $\alpha$  \\ \hline &&&\\[-1em]
		$A^*$	& $-{\alpha Kb\over\theta}$ & $\sqrt{{\theta}{\sigma_{\bf x}^2\over\sigma_{\bf v}^2}}\left(A^*\sigma_{\bf x}^2\sigma_{\bf w}^2 + \sigma_{\bf x}^2\sigma_{\bf v}^2\right)$ & ${\alpha^2 K^2b\over\theta}$ & ${A^*\sigma_{\bf w}^2 \over A^*\sigma_{\bf w}^2+\sigma_{\bf v}^2}$ \\ \hline &&&\\[-1em]
		$-\sqrt{\sqrt{\sigma_{\bf v}^2\over\theta\sigma_{\bf x}^2 }-{\sigma_{\bf v}^2\over\sigma_{\bf x}^2 }}$	& $-{b \sqrt{\sqrt{\theta\sigma_{\bf x}^2\over \sigma_{\bf v}^2}-\theta}\over\theta}$ & $\sqrt{\sqrt{\theta\sigma_{\bf x}^2\over \sigma_{\bf v}^2}-\theta}$ & ${b\left({\sqrt{\theta\sigma_{\bf x}^2\over \sigma_{\bf v}^2}-\theta}\right)\over\theta}$ & $1$ \\ \hline 
	\end{tabular}%
	\end{table}

	Furthermore, for any value of $\theta$, the following also forms an affine Nash equilibrium:
	\begin{align*}
	A=0 \,,\quad C=0 \,,\quad K = {\sigma_{\bf x}^2 \over \sigma_{\bf x}^2+\sigma_{\bf w}^2} \,,\quad L=0 \,,\quad \alpha = 0\,.
	\end{align*}
	\end{itemize}
\end{thm}
\IEEEproof
See Appendix~\ref{appendixaffineNash}.
\endIEEEproof

\begin{remark}
Similar to the single-stage case, affine strategies constitute an invariant subspace under best response maps for multi-stage Nash equilibria. In particular, the first and the second parts of Theorem~\ref{thm:affineNash} can be extended to the multi-stage scenario. However, since the number of equations and unknowns increase quadratically, an explicit analysis as in the third part of Theorem~\ref{thm:affineNash} becomes infeasible.
\end{remark}

\section{Conclusion}

In this paper, we studied Nash and Stackelberg equilibria for single-stage and multi-stage quadratic signaling games with channel combining\&utilization at the decoder. We established qualitative (e.g. 
linearity and informativeness) and quantitative properties (on linearity or explicit computation) of Nash and Stackelberg equilibria under misaligned objectives.

Our model has many possible interesting extensions. Of
particular interest are the case when there is a hard power constraint for the encoder and the analysis of steady state equilibria. Scenarios with more general alternative channels/encoders, or with more general objective functions are also under consideration.

\appendices

\section{Proof of Theorem~\ref{thm:affineDecoder}} \label{appendixaffineDecoder} 
	For the given encoder strategy ${\bf m}=\gamma^e({\bf x})=A{\bf x}$, the decoder input is 
\begin{align*}
{\bf r} = (\alpha A+1-\alpha){\bf x}+\alpha {\bf v}+(1-\alpha){\bf w}  
\end{align*}
when the decoder adjusts the time-sharing parameter $\alpha$ of the channels. Then, the optimal decoder strategy is $\gamma^d({\bf r})=\hat{{\bf x}}={\bf E}[{\bf x}|{\bf r}]$, which can be expressed as
\begin{align}
\gamma^d({\bf r}) = \underbrace{{(\alpha A+1-\alpha)\sigma_{\bf x}^2 \over (\alpha A+1-\alpha)^2\sigma_{\bf x}^2 + \alpha^2\sigma_{\bf v}^2 + (1-\alpha)^2\sigma_{\bf w}^2}}_{\triangleq K}{\bf r} \,.
\label{eq:optDecoderEstCombined}
\end{align}
Then, the corresponding decoder cost (namely, the estimation error) is
\begin{align}
J^d &= {\bf E}[({\bf x}-K{\bf r})^2] = ((\alpha A+1-\alpha)K-1)^2\sigma_{\bf x}^2  + \alpha^2K^2\sigma_{\bf v}^2 + (1-\alpha)^2K^2\sigma_{\bf w}^2 \,.
\label{eq:decoderCostCombined}
\end{align}

The decoder is trying to minimize $J^d$ by adjusting both $K$ and $\alpha$. Let $t\triangleq\alpha K$ and $u\triangleq (1-\alpha) K$. Then, the decoder cost becomes $J^d = (At+u-1)^2\sigma_{\bf x}^2 + t^2\sigma_{\bf v}^2 + u^2\sigma_{\bf w}^2$.
Since the Hessian matrix $\mathbb{H}=\begin{bmatrix} {\partial^2 J^d\over\partial t^2} & {\partial^2 J^d\over\partial t\partial u} \\
{\partial^2 J^d\over\partial u\partial t} & {\partial^2 J^d\over\partial u^2} \end{bmatrix}$
is positive semi-definite, $J^d$ is a convex function of $t$ and $u$. Thus, at the optimum point, i.e., when ${\partial J^d \over \partial t}={\partial J^d \over \partial u}=0$, we obtain $\alpha = {A\sigma_{\bf w}^2 \over A\sigma_{\bf w}^2+\sigma_{\bf v}^2}$ and $K = {A\sigma_{\bf x}^2\sigma_{\bf w}^2 + \sigma_{\bf x}^2\sigma_{\bf v}^2\over A^2\sigma_{\bf x}^2\sigma_{\bf w}^2+\sigma_{\bf x}^2\sigma_{\bf v}^2+\sigma_{\bf w}^2\sigma_{\bf v}^2}$.
By inserting these into \eqref{eq:decoderCostCombined}, we obtain $J^d = {\sigma_{\bf x}^2 \over A^2{\sigma_{\bf x}^2\over\sigma_{\bf v}^2}+{\sigma_{\bf x}^2\over\sigma_{\bf w}^2}+1}$.

However, note that 
when $A<0$, the optimal $\alpha$ lies outside of its feasible region $[0,1]$. Thus, the extreme values of this closed interval should be compared for $A<0$.

Let $\alpha=0$. Then, the optimal decoder is $\gamma^d({\bf r})={\sigma_{\bf x}^2 \over \sigma_{\bf x}^2+\sigma_{\bf w}^2}{\bf r}$ by \eqref{eq:optDecoderEstCombined}, which results in the decoder cost $J^d={\sigma_{\bf x}^2\sigma_{\bf w}^2 \over \sigma_{\bf x}^2+\sigma_{\bf w}^2}$ by \eqref{eq:decoderCostCombined}.

Now let $\alpha=1$, i.e., a point-to-point communication scenario is considered. Then, the optimal decoder is $\gamma^d({\bf r})={A\sigma_{\bf x}^2 \over A^2\sigma_{\bf x}^2+\sigma_{\bf v}^2}{\bf r}$ by \eqref{eq:optDecoderEstCombined}, which corresponds to the decoder cost $J^d={\sigma_{\bf x}^2\sigma_{\bf v}^2 \over A^2\sigma_{\bf x}^2+\sigma_{\bf v}^2}$ by \eqref{eq:decoderCostCombined}.

Then, the following comparison can be made to find the optimal decoder for $A<0$:
\begin{align}
{\sigma_{\bf x}^2\sigma_{\bf w}^2 \over \sigma_{\bf x}^2+\sigma_{\bf w}^2} \overset{\alpha=0}{\underset{\alpha=1}{\lesseqgtr}} {\sigma_{\bf x}^2\sigma_{\bf v}^2 \over A^2\sigma_{\bf x}^2+\sigma_{\bf v}^2} &\Rightarrow A^2 \overset{\alpha=0}{\underset{\alpha=1}{\lesseqgtr}} {\sigma_{\bf v}^2\over\sigma_{\bf w}^2} \,.
\end{align}
Hence, $-\sqrt{\sigma_{\bf v}^2\over\sigma_{\bf w}^2}\leq A\leq0$ corresponds to the case with $\alpha=0$, and $A\leq-\sqrt{\sigma_{\bf v}^2\over\sigma_{\bf w}^2}$ corresponds to the case with $\alpha=1$.\\
This completes the derivation. \qed

\section{Proof of Theorem~\ref{thm:linearOpt}} \label{appendixlinearOpt} 
	
In the proof, we first obtain information theoretic lower bound on the estimation error, then show that this lower bound is achieved when the players jointly utilize linear strategies.

Since the decoder's received signal is ${\bf r}=\alpha({\bf m}+{\bf v}) + (1-\alpha)({\bf x}+{\bf w})$, its power can be expressed as
\begin{align*}
{\bf E}[{\bf r}^2] 
&= \underbrace{\alpha^2 P + (1-\alpha)^2\sigma_{\bf x}^2 + 2\alpha(1-\alpha){\bf E}[{\bf m}{\bf x}]}_{\text{signal power}} +\underbrace{\alpha^2\sigma_{\bf v}^2 + (1-\alpha)^2\sigma_{\bf w}^2}_{\text{noise power}} \,.
\end{align*}
Since the channels are additive Gaussian, a corresponding (combined) channel/information capacity $C$ between ${\bf x}$ and ${\bf r}$ can be represented as
\begin{align}
C &= \sup I({\bf x};{\bf r}) = \frac{1}{2} \log_2\left(1+{\alpha^2 P + (1-\alpha)^2\sigma_{\bf x}^2 + 2\alpha(1-\alpha){\bf E}[{\bf m}{\bf x}] \over \alpha^2\sigma_{\bf v}^2 + (1-\alpha)^2\sigma_{\bf w}^2}\right) \,. \label{eq:capXR}
\end{align}
Then, the lower bound on the estimation error can be derived as follows:
\begin{align} \label{eq:scalarITinequalities}
I({\bf x};{\bf r}) &= h({\bf x})-h({\bf x}|{\bf r}) = h({\bf x})-h({\bf x}-{\bf E}[{\bf x}|{\bf r}]|{\bf r}) \nn\\
&\geq h({\bf x})-h({\bf x}-{\bf E}[{\bf x}|{\bf r}]) \overset{(a)}{\geq} {1\over 2}\log_2(2\pi \mathrm{e} \sigma_{\bf x}^2)-{1\over 2}\log_2(2\pi \mathrm{e} J^d) \nn\\
&\Rightarrow I({\bf x};{\bf r}) \geq {1\over 2}\log_2({\sigma_{\bf x}^2 \over J^d}) \nn\\
&\Rightarrow J^d \geq \sigma_{\bf x}^2 2^{-2I({\bf x};{\bf r})} \geq \sigma_{\bf x}^2 2^{-2 \sup I({\bf x};{\bf r})} 
\overset{(b)}{=} \sigma_{\bf x}^2 2^{-2 \frac{1}{2} \log_2\big(1+{\alpha^2 P + (1-\alpha)^2\sigma_{\bf x}^2 + 2\alpha(1-\alpha){\bf E}[{\bf m}{\bf x}] \over \alpha^2\sigma_{\bf v}^2 + (1-\alpha)^2\sigma_{\bf w}^2}\big)} \nn\\
&\qquad\;\; =\frac{\sigma_{\bf x}^2}{1+{\alpha^2 P + (1-\alpha)^2\sigma_{\bf x}^2 + 2\alpha(1-\alpha){\bf E}[{\bf m}{\bf x}] \over \alpha^2\sigma_{\bf v}^2 + (1-\alpha)^2\sigma_{\bf w}^2}} \overset{(c)}{\geq} \frac{\sigma_{\bf x}^2}{{P\over\sigma_{\bf v}^2}+{\sigma_{\bf x}^2\over\sigma_{\bf w}^2}+1} \,.
\end{align} 
Here, (a) holds since the differential entropy is $h({\bf x}) = {1\over 2}\log_2(2\pi \mathrm{e} \sigma_{\bf x}^2)$ for a Gaussian source ${\bf x}$ and maximum $h({\bf x}-{\bf E}[{\bf x}|{\bf r}])$ is achieved when ${\bf x}-{\bf E}[{\bf x}|{\bf r}]$ is Gaussian, (b) follows from \eqref{eq:capXR}, and (c) holds for $0<\alpha<1$ due to the following inequalities:
\begin{align*}
1+{\alpha^2 P + (1-\alpha)^2\sigma_{\bf x}^2 + 2\alpha(1-\alpha){\bf E}[{\bf m}{\bf x}] \over \alpha^2\sigma_{\bf v}^2 + (1-\alpha)^2\sigma_{\bf w}^2} 
&\overset{(a)}{\leq} 1+{\alpha^2 P + (1-\alpha)^2\sigma_{\bf x}^2 + 2\alpha(1-\alpha)\sqrt{P\sigma_{\bf x}^2} \over \alpha^2\sigma_{\bf v}^2 + (1-\alpha)^2\sigma_{\bf w}^2} \\
&\overset{(b)}{\leq} 1+{\alpha^2 P\over\alpha^2\sigma_{\bf v}^2}+{(1-\alpha)^2\sigma_{\bf x}^2\over(1-\alpha)^2\sigma_{\bf w}^2} \,,
\end{align*} 
where (a) holds due to the Cauchy-Schwarz inequality and (b) holds\footnote{Note that ${a \over c} + {b \over d}-{(\sqrt{a}+\sqrt{b})^2 \over c+d}={ad+bc \over cd}-{a+b+2\sqrt{ab} \over c+d}={ad^2+bc^2-2\sqrt{ab}cd \over cd(c+d)} = {(\sqrt{a}d-\sqrt{b}c)^2\over cd(c+d)}\geq0$.} since ${(\sqrt{a}+\sqrt{b})^2 \over c+d} \leq {a \over c} + {b \over d}$ for positive $a,b,c,d$ with $a=\alpha^2P$, $b=(1-\alpha)^2\sigma_{\bf x}^2$, $c=\alpha^2\sigma_{\bf v}^2$, and $d=(1-\alpha)^2\sigma_{\bf w}^2$. 

In \eqref{eq:scalarITinequalities}, the first inequality is tight iff ${\bf x}$ and ${\bf r}$ are jointly Gaussian, which is satisfied for a linear encoder, whereas the second inequality (i.e., (c) of \eqref{eq:scalarITinequalities}) holds with equality for $0<\alpha<1$ when $\sqrt{a}d=\sqrt{b}c \Rightarrow \alpha\sqrt{P}(1-\alpha)^2\sigma_{\bf w}^2=(1-\alpha)\sigma_{\bf x}\alpha^2\sigma_{\bf v}^2 \Rightarrow \sqrt{P}={\alpha\over1-\alpha}\sigma_{\bf x}{\sigma_{\bf v}^2\over\sigma_{\bf w}^2}$. Since $\alpha={A\sigma_{\bf w}^2 \over A\sigma_{\bf w}^2+\sigma_{\bf v}^2}$ for a linear encoder ${\bf m}=\gamma^e({\bf x})=A{\bf x}$ as shown in Theorem~\ref{thm:affineDecoder}, we obtain ${\bf E}[{\bf m}^2]=P=A^2\sigma_{\bf x}^2$, which is consistent with a linear encoder case. Note that for $\alpha=0$, (c) in \eqref{eq:scalarITinequalities} reduces to $\frac{\sigma_{\bf x}^2}{{\sigma_{\bf x}^2\over\sigma_{\bf w}^2}+1}$ with equality, and for $\alpha=1$, (c) in \eqref{eq:scalarITinequalities} reduces to $\frac{\sigma_{\bf x}^2}{{P\over\sigma_{\bf v}^2}+1}$ with equality. Thus, the information theoretic lower bound on the estimation error is $\frac{\sigma_{\bf x}^2}{{P\over\sigma_{\bf v}^2}+{\sigma_{\bf x}^2\over\sigma_{\bf w}^2}+1}$ and it is achievable only for jointly linear encoder and decoder with $A>0$ and $0<\alpha<1$.
This completes the derivation. \qed

\section{Proof of Theorem~\ref{thm:linearEq}} \label{appendixlinearEq} 
Due to Theorem~\ref{thm:linearOpt} and Remark~\ref{rem:encCost}, the encoder cost is lower bounded by
\begin{align*}
J^e \geq \frac{\sigma_{\bf x}^2}{{P\over\sigma_{\bf v}^2}+{\sigma_{\bf x}^2\over\sigma_{\bf w}^2}+1} + \theta P^2 + b^2 \triangleq J^e_{\text{LB}}\,,
\end{align*}
where $P\triangleq{\bf E}[{\bf m}^2]$. Note that the lower bound $J^e_{\text{LB}}$ is achievable when the encoder use linear strategy. The first and second order derivatives of the lower bound $J^e_{\text{LB}}$ are
\begin{align*}
{\mathrm{d}J^e_{\text{LB}}\over\mathrm{d}P} &= -{{\sigma_{\bf x}^2\over\sigma_{\bf v}^2} \over \left({P\over\sigma_{\bf v}^2}+{\sigma_{\bf x}^2\over\sigma_{\bf w}^2}+1\right)^2} + \theta \geq \theta - {{\sigma_{\bf x}^2\over\sigma_{\bf v}^2} \over \left({\sigma_{\bf x}^2\over\sigma_{\bf w}^2}+1\right)^2} \,,\\
{\mathrm{d}^2J^e_{\text{LB}}\over\mathrm{d}P^2} &= {2{\sigma_{\bf x}^2\over\sigma_{\bf v}^4} \over \left({P\over\sigma_{\bf v}^2}+{\sigma_{\bf x}^2\over\sigma_{\bf w}^2}+1\right)^3} > 0  \,.
\end{align*}
If $\theta\geq{{\sigma_{\bf x}^2\over\sigma_{\bf v}^2} \over \left({\sigma_{\bf x}^2\over\sigma_{\bf w}^2}+1\right)^2}$, then ${\mathrm{d}J^e_{\text{LB}}\over\mathrm{d}P}\geq0$, which implies that $J^e_{\text{LB}}$ is an increasing function of $P$, thus $P$ should be selected as $P=0$, i.e., the encoder does not transmit any message. 

Otherwise, i.e., if $\theta<{{\sigma_{\bf x}^2\over\sigma_{\bf v}^2} \over \left({\sigma_{\bf x}^2\over\sigma_{\bf w}^2}+1\right)^2}$, the lower bound can be minimized at the critical point , ${\mathrm{d}J^e_{\text{LB}}\over\mathrm{d}P}=0$, which implies $P^*=\sqrt{1\over\theta{\sigma_{\bf x}^2\over\sigma_{\bf v}^2}}-{{\sigma_{\bf x}^2\over\sigma_{\bf w}^2}+1\over{\sigma_{\bf x}^2\over\sigma_{\bf v}^2}}$. Since $J^e_{\text{LB}}$ is achievable for a linear encoder ${\bf m}=\gamma^e({\bf x})=A{\bf x}$, since $P^*\triangleq{\bf E}[(A^*{\bf x})^2]=(A^*)^2$, the optimal $A$ is obtained as $A^*=\sqrt{\sqrt{1\over\theta{\sigma_{\bf x}^2\over\sigma_{\bf v}^2}}-{{\sigma_{\bf x}^2\over\sigma_{\bf w}^2}+1\over{\sigma_{\bf x}^2\over\sigma_{\bf v}^2}}}$.
\begin{align}
\min\limits_{P} J^e_{\text{LB}} = \min\limits_{P} {\sigma_{\bf x}^2 \over {P\over\sigma_{\bf v}^2}+{\sigma_{\bf x}^2\over\sigma_{\bf w}^2}+1}+ \theta P  \,.
\label{eq:affineEncoderSoft}
\end{align}

Otherwise, i.e., if we have $\theta<{{\sigma_{\bf x}^2\over\sigma_{\bf v}^2} \over \left({\sigma_{\bf x}^2\over\sigma_{\bf w}^2}+1\right)^2}$, the critical point ${\mathrm{d}J^e\over\mathrm{d}P}=0$ when $P=\sqrt{1\over\theta{\sigma_{\bf x}^2\over\sigma_{\bf v}^2}}-{{\sigma_{\bf x}^2\over\sigma_{\bf w}^2}+1\over{\sigma_{\bf x}^2\over\sigma_{\bf v}^2}}$. Thus, the optimal $A$ is $A=\sqrt{\sqrt{1\over\theta{\sigma_{\bf x}^2\over\sigma_{\bf v}^2}}-{{\sigma_{\bf x}^2\over\sigma_{\bf w}^2}+1\over{\sigma_{\bf x}^2\over\sigma_{\bf v}^2}}}$.\\
This completes the derivation. \qed

\section{Proof of Theorem~\ref{thm:dynamic_sol}} \label{appendixdynamic_sol} 
From the system in Fig. \ref{fig:dynamic_setup1}, we know that the observations process $\{{\bf r}_t:~t\in\mathbb{N}_0^n\}$ is given by
\begin{align}
\begin{split}
{\bf r}_t&=\alpha_t{\bf y}_t+(1-\alpha_t){\bf z}_t\\
&\stackrel{(i)}=\alpha_t({\bm \gamma}_t^\epsilon({\bf x}^t)+{\bf v}_t)+(1-\alpha_t)({\bf x}_t+{\bf w}_t)\\
&=\alpha_t(A_t{\bf x}_t+{\bf v}_t)+(1-\alpha_t)({\bf x}_t+{\bf w}_t)\\
&=(\alpha_tA_t+1-\alpha_t){\bf x}_t+\alpha_t{\bf v}_t+(1-\alpha_t){\bf w}_t,~t\in\mathbb{N}_0^n,
\end{split}\label{observations_process}
\end{align}  
where $(i)$ follows from the realization in Fig. \ref{fig:dynamic_setup1} and \eqref{csi}. Moreover, since the minimum error at the decoder at each instant of time is ${\bf E}[({\bf x}_t-\hat{\bf x}_{t|t})^2]$, then, the decoder's cost can be modified as follows:
\begin{align}
\begin{split}
J_{t}^d&={\bf E}[({\bf x}_t-\hat{\bf x}_{t|t})^2]\\
&\stackrel{(i)}={\bf E}[({\bf x}_t-{\bf x}_{t|t-1}-{\bf k_t}({\bf r}_t-(\alpha_tA_t+1-\alpha_t)\hat{\bf x}_{t|t-1}))^2]\\
&\stackrel{\mathclap{(ii)}}{=}{\bf E}[({\bf x}_t-{\bf x}_{t|t-1}-{\bf k}_t((\alpha_tA_t+1-\alpha_t)({\bf x}_t-\hat{\bf x}_{t|t-1})+\alpha_t{\bf v}_t+(1-\alpha_t){\bf w}_t)^2]\\
&={\bf E}[(1-{\bf k}_t(\alpha_tA_t+1-\alpha_t)({\bf x}_t-\hat{\bf x}_{t|t-1})-{\bf k}_t\alpha_t{\bf v}_t-(1-\alpha_t){\bf k}_t{\bf w}_t)^2]\\
&=\left[(1-{\bf k}_t(\alpha_tA_t+1-\alpha_t))^2\Sigma_{t|t-1}+{\bf k}^2_t\alpha^2_t\sigma^2_{{\bf v}_t}+(1-\alpha_t)^2{\bf k}^2_t\sigma^2_{{\bf w}_t}\right]
\end{split}\label{decoders_cost}
\end{align} 
where $(i)$ follows from \eqref{kf}; $(ii)$ follows by substituting in our expression  \eqref{observations_process} and after some simple calculations. 

\begin{remark}\label{remark:equivalent_form} The decoder's cost in \eqref{decoders_cost} although written in different form, is precisely $\Sigma_{t|t}\geq{0}$ because the conditional variance $\Sigma_{t|t}$ is equal to the unconditional in KF algorithm.
\end{remark}

\noindent{\bf Optimization Problem.} We will solve the decoder's optimization problem in \eqref{total_cost} forward in time, starting at time stage zero and moving forward to a fixed time stage $n$. To do it, first we re-formulate it as follows:
\begin{align}
J_{\avertot}^d=\frac{1}{n+1}\min_{\alpha_n}\left\{\min_{\alpha_{n-1}}\left\{\ldots\min_{\alpha_{1}}\left\{\left\{\min_{\alpha_{0}}J_0\right\}+J_1\right\}+\ldots+J_{n-1}\right\}+J_n\right\},\label{nested_opt}
\end{align}
where 
\begin{align}
J_0=&(1-{\bf k}_0(\alpha_0A_0+1-\alpha_0))^2\Sigma_{0|-1}+{\bf k}^2_0\alpha^2_0\sigma^2_{{\bf v}_0}+(1-\alpha_0)^2{\bf k}^2_0\sigma^2_{{\bf w}_0},~\Sigma_{0|-1}=\sigma^2_{{\bf x}_0}\,,\label{initial}\\
J_t=&(1-{\bf k}_t(\alpha_tA_t+1-\alpha_t))^2\Sigma_{t|t-1}+{\bf k}^2_t\alpha^2_t\sigma^2_{{\bf v}_t}+(1-\alpha_t)^2{\bf k}^2_t\sigma^2_{{\bf w}_t}\,.\label{cost_to_go}
\end{align}

We first consider $t=0$. Using the formulation in \eqref{nested_opt}, we want to optimize 
\begin{align}
\min_{\alpha_{0}}J_0.\label{initial_optimization}
\end{align}
Observe that from \eqref{decoders_cost}, by optimizing w.r.t. $\alpha_0$, it is the same as optimizing w.r.t $({\bf k}_0, \alpha_0)$ because ${\bf k}_0$ depends on $\alpha_0$. Hence, we can re-write \eqref{initial_optimization} as 
\begin{align}
\min_{\alpha_{0},~{\bf k}_0}(1-{\bf k}_0(\alpha_0A_0+1-\alpha_t))^2\Sigma^*_{0|-1}+{\bf k}^2_0\alpha^2_0\sigma^2_{{\bf v}_0}+(1-\alpha_0)^2{\bf k}^2_0\sigma^2_{{\bf w}_0},~\Sigma^*_{0|-1}=\sigma^2_{{\bf x}_0}.\label{final_form_t0}
\end{align}
To solve \eqref{final_form_t0}, we first show that it is convex. To do it, we first introduce the auxiliary variables 
\begin{align}
{\bm \phi}_0=\alpha_0{\bf k}_0,~~~{\bm \upsilon}_0=(1-\alpha_0){\bf k}_0.\label{aux_var_t0}
\end{align}
For the choice of \eqref{aux_var_t0}, \eqref{final_form_t0} can be simplified to:
\begin{align}
\min_{{\bm \phi}_{0},~{\bm \upsilon}_0}(1-A_0{\bm \phi}_0-{\bm \upsilon}_0)^2\sigma^2_{{\bf x}_0}+{\bm \phi}_0^2\sigma^2_{{\bf v}_0}+{\bm \upsilon}^2_0\sigma^2_{{\bf w}_0}.\label{final_form_t0_1}
\end{align}
The Hessian matrix that corresponds to the objective function of \eqref{final_form_t0_1}, hereinafter denoted by $\mathbb{H}_0$, can be found as follows: 
\begin{align}
\begin{split}
{\partial J_0^d \over \partial {\bm \phi}_0} &= 2(A_0{\bm \phi}_0+{\bm \upsilon}_0-1)A_0\sigma^2_{{\bf x}_0}+2{\bm \phi}_0\sigma^2_{{\bf v}_0}\,,\quad\qquad\qquad {\partial^2 J_0^d \over \partial {\bm \phi}^2_0}= 2A^2_0\sigma^2_{{\bf x}_0}+2\sigma^2_{{\bf v}_0}\\
{\partial J_0^d \over \partial {\bm \upsilon}_0} &= 2(A_0{\bm \phi}_0+{\bm \upsilon}_0-1)\sigma^2_{{\bf x}_0} + 2{\bm \upsilon}_0\sigma^2_{{\bf w}_0}\,,\quad\quad\qquad\qquad {\partial^2 J_0^d \over \partial {\bm \upsilon}^2_0}= 2\sigma^2_{{\bf x}_0}+2\sigma^2_{{\bf w}_0}\\ 
{\partial J_0^d \over \partial {\bm \upsilon}_0\partial{\bm \phi}_0}&={\partial J_0^d \over \partial{\bm \phi}_0\partial {\bm \upsilon}_0}=2A_0\sigma^2_{{\bf x}_0}.
\end{split}\label{partial_der_0}
\end{align}
Based on \eqref{partial_der_0}, the Hessian matrix $\mathbb{H}_0$ is given as follows
\begin{align}
\mathbb{H}_0=\begin{bmatrix} {\partial^2 J_0^d\over\partial {\bm \phi}_0^2} & {\partial^2 J_0^d\over\partial {\bm \upsilon}_0\partial {\bm \phi}_0} \\
{\partial^2 J_0^d\over\partial {\bm \phi}_0\partial {\bm \upsilon}_0} & {\partial^2 J_0^d\over\partial {\bm \phi}_0^2} \end{bmatrix}=\begin{bmatrix} 2A^2_0\sigma^2_{{\bf x}_0} +2\sigma^2_{{\bf v}_0} & 2A_0\sigma^2_{{\bf x}_0} \\
2A_0\sigma^2_{{\bf x}_0} & 2\sigma^2_{{\bf x}_0} +2\sigma^2_{{\bf w}_0}
\end{bmatrix}.\label{hessian_matrix_0}
\end{align}
It can be easily checked that for any $A_0$, the eigenvalues of  $\mathbb{H}_0$ are non-negative, hence the matrix is positive semi-definite. This in turn implies that $J_0^d$ is jointly convex on $({\bm \phi}_0, {\bm \upsilon}_0)$.\\
Therefore, the optimal solution $\alpha^*_0$ is as follows:
\begin{align}
{\partial J_0^d \over \partial {\bm \phi}_0}={\partial J_0^d \over \partial {\bm \upsilon}_0}=0 &\stackrel{\eqref{partial_der_0}}\Rightarrow -{{\bm \phi}_0\sigma^2_{{\bf v}_0}\over A_0} =-{\upsilon}_0\sigma^2_{{\bf w}_0} \Rightarrow {\sigma^2_{{\bf v}_0}\over A_0\sigma^2_{{\bf w}_0}}=\frac{{\bm \upsilon}_0}{{\bm \phi}_0} \stackrel{\eqref{aux_var_t0}}= {(1-\alpha^*_0){\bf k}_0 \over \alpha^*_0 {\bf k}_0} \nn\\
&\Rightarrow \alpha^*_0 = {A_0\sigma^2_{{\bf w}_0} \over A_0\sigma^2_{{\bf w}_0}+\sigma^2_{{\bf v}_0}} \stackrel{\eqref{kf}}\Rightarrow {\bf k}^*_0 = {A_0\sigma^2_{{\bf w}_0} + \sigma^2_{{\bf v}_0}\over A_0^2\sigma^2_{{\bf w}_0}+\sigma^2_{{\bf v}_0}+{\sigma^2_{{\bf w}_0}\sigma^2_{{\bf v}_0}\over\sigma^2_{{\bf x}_0}}}
= {(A_0\sigma^2_{{\bf w}_0}+\sigma^2_{{\bf v}_0})\sigma^2_{{\bf x}_0}\over (A_0^2\sigma^2_{{\bf w}_0}+\sigma^2_{{\bf v}_0})\sigma^2_{{\bf x}_0}+\sigma^2_{{\bf w}_0}\sigma^2_{{\bf v}_0}}.
\label{optimal_sol:t0}
\end{align} 
Substituting ($\alpha_0^*, {\bf k}_0^*$) obtained in \eqref{optimal_sol:t0} to \eqref{final_form_t0}, we obtain $J_0^{d,*}=\frac{\Sigma_{0|-1}^*\sigma^2_{{\bf w}_0}\sigma^2_{{\bf v}_t}}{(A_t^2\sigma^2_{{\bf w}_0}+\sigma^2_{{\bf v}_0})\Sigma_{0|-1}^*+\sigma^2_{{\bf w}_0}\sigma^2_{{\bf v}_0}}$.\\

Similar to the single-stage case, note that when $A_0<0$, the optimal $\alpha^*_0$ lies outside the feasible region $[0,1]$. Thus, the extreme values of this closed interval should be compared for $A_0<0$. This is done next.

Let $\alpha_0^*=0$. Then, the optimal decoder is $\gamma_0^d({\bf r_0})={\Sigma^*_{0|-1}\over \Sigma_{0|-1}+\sigma_{{\bf w}_0}^2}{\bf r_0}$ by \eqref{kf}, which means that the decoder's cost becomes $J^{d,*}_0={\Sigma_{0|-1}\sigma_{{\bf w}_0}^2 \over \Sigma_{0|-1}+\sigma_{{\bf w}_0}^2}$ again from  \eqref{kf}.

Now let $\alpha_0^*=1$, i.e., a point-to-point communication scenario is considered without side information. Then, the optimal decoder is $\gamma_0^d({\bf r}_0)={A_0\Sigma_{0|-1} \over A_0^2\Sigma_{0|-1}+\sigma_{{\bf v}_0}^2}{\bf r}_0$ by \eqref{kf}, which yields a decoder's cost $J^{d,*}_0={\Sigma_{0|-1}\sigma_{{\bf v}_0}^2 \over A_0^2\Sigma_{0|-1}+\sigma_{{\bf v}_0}^2}$ by \eqref{kf}.

Next, we proceed to $t=1$. Again, using the formulation in \eqref{nested_opt}, this corresponds to the optimization problem 
\begin{align}
\min_{\alpha_{1}}\left\{J_0^{d,*}+J_1^d\right\}\stackrel{(a)}\equiv\min_{\alpha_{1}}J_1^d,\label{initial_optimization_t1}
\end{align}
where $(a)$ follows because $J_0^{d,*}$ is a constant as it is already optimized in time stage  $t=0$.\\
Observe from \eqref{decoders_cost} that by optimizing w.r.t. $\alpha_1$, is the same as optimizing w.r.t $({\bf k}_1, \alpha_1)$ because ${\bf k}_1$ depends on $\alpha_1$. Hence, we can re-write \eqref{initial_optimization} as 
\begin{align}
\min_{\alpha_{1},~{\bf k}_1}(1-{\bf k}_1(\alpha_1A_1+1-\alpha_t))^2\Sigma^*_{1|0}+{\bf k}^2_1\alpha^2_1\sigma^2_{{\bf v}_1}+(1-\alpha_1)^2{\bf k}^2_1\sigma^2_{{\bf w}_1},\label{final_form_t1}
\end{align}
where $\Sigma^*_{1|0}=\beta^2_0\Sigma^*_{0|0}+\sigma^2_{{\bf n}_0}$ by \eqref{kf} hence it is independent of $\alpha_1$. The latter observation stems from the fact that $\Sigma_{0|0}=J_0$ (see Remark \ref{remark:equivalent_form}). Therefore, the procedure, is precisely the same as in time stage $t=0$, with $\sigma^2_{{\bf x}_0}$ replaced by $\Sigma^*_{1|0}$. The final result is given in Table~\ref{tab:optDecoderMulti} when $t=1$. \\
Suppose that for $t=n-1$, the solution is given by $J_{n-1}^{d,*}$ in Table~\ref{tab:optDecoderMulti}. Then, for $t=n$ following the approach of time stage $t=0$, the solution will be given by $J_{n}^{d,*}$ in Table~\ref{tab:optDecoderMulti}.\\ 
Clearly, the average total cost of the decoder in \eqref{total_cost} is the average total time stages of all individual optimal decoder's costs. 
This completes the derivation. \qed

\section{Proof of Proposition~\ref{prop:1}} \label{appendixprop1} 
Observe that if the linear encoder is of the class \eqref{linear_encoder_memory}, then, by definition, the innovations process is obtained as follows:
\begin{align}
\begin{split}
{\bf I}_t&\triangleq{\bf r}_t-{\bf E}[{\bf r}_t|{\bf r}^{t-1}],~t\in\mathbb{N}_0^n\\
&=\alpha_tA_t({\bf x}_t-\hat{\bf x}_{t|t-1})+(1-\alpha_t){\bf x}_t+\alpha_t{\bf v}_t+(1-\alpha_t){\bf w}_t\\
&\qquad\qquad-{\bf E}[\alpha_tA_t({\bf x}_t-\hat{\bf x}_{t|t-1})+(1-\alpha_t){\bf x}_t+\alpha_t{\bf v}_t+(1-\alpha_t){\bf w}_t|{\bf r}^{t-1}]\\
&\stackrel{\mathclap{(i)}}=\alpha_tA_t({\bf x}_t-\hat{\bf x}_{t|t-1})+(1-\alpha_t){\bf x}_t+\alpha_t{\bf v}_t+(1-\alpha_t){\bf w}_t\\
&\qquad\qquad-\alpha_tA_t{\bf x}_{t|t-1}-\alpha_tA_t{\bf E}[\hat{\bf x}_{t|t-1}|{\bf r}^{t-1}]+(1-\alpha_t){\bf x}_{t|t-1}\\
&\stackrel{\mathclap{(ii)}}=\alpha_tA_t({\bf x}_t-\hat{\bf x}_{t|t-1})+(1-\alpha_t){\bf x}_t+\alpha_t{\bf v}_t+(1-\alpha_t){\bf w}_t-(1-\alpha_t)\hat{\bf x}_{t|t-1}\\
&=\text{innovations in \eqref{kf}},
\end{split}
\end{align}
where $(i)$ follows from the fact that the expectation is a linear operator, $C_t$ is a constant and that the noise process $\{{\bf v}_t:~t\in\mathbb{N}_0^n\}$ and  $\{{\bf w}_t:~t\in\mathbb{N}_0^n\}$ are zero mean mutually independent processes independent of everything; $(ii)$ follows from the tower property of conditional expectation or simply because $\hat{\bf x}_{t|t-1}$ is ${\bf r}^{t-1}$-measurable. Since the innovations process generates the same information with a linear memoryless encoder, then, the results obtained in Table \ref{tab:optDecoderMulti} will also applied for this class of linear encoders. \qed

\section{Proof of Theorem~\ref{theorem:optimality_linear_policies_dynamic}} \label{appendixoptimality_linear_policies_dynamic} 
The proof is obtained using first an information theoretic lower bound on the estimation error, and then, by showing that this lower bound is achievable when the players jointly utilize linear strategies.
\par Since the decoder's received signal at each instant of time is ${\bf r}_t=\alpha_t({\bf m}_t+{\bf v}_t) + (1-\alpha_t)({\bf x}_t+{\bf w}_t)$, the conditional mean and conditional variance (power) of $\{{\bf r}_t:~t\in\mathbb{N}_0^n\}$ are as follows:\footnote{Recall that conditional variance is equivalent with the unconditional for jointly Gaussian processes.}
\begin{align}
{\bf E}[{\bf r}_t|{\bf r}^{t-1}]&=\alpha_t\hat{\bf m}_{t|t-1}+(1-\alpha_t)\hat{\bf x}_{t|t-1}\label{cond_mean_r}\\
{\bf E}\left[({\bf r}_t-{\bf E}[{\bf r}_t|{\bf r}^{t-1}])^2|{\bf r}^{t-1}\right]\equiv{\bf E}\left[({\bf r}_t-{\bf E}[{\bf r}_t|{\bf r}^{t-1}])^2\right]&=\alpha^2_t{\bf E}\left[({\bf m}_t-\hat{\bf m}_{t|t-1})^2\right]+(1-\alpha_t)^2{\bf E}\left[({\bf x}_t-\hat{\bf x}_{t|t-1})^2\right]+\alpha_t^2{\bf E}\left[{\bf v}_t^2\right]\nonumber\\
&\quad + (1-\alpha_t)^2{\bf E}\left[{\bf w}_t^2\right] +2\alpha_t(1-\alpha_t){\bf E}\left[({\bf m}_t-\hat{\bf m}_{t|t-1})({\bf x}_t-\hat{\bf x}_{t|t-1})|{\bf r}^{t-1}\right]\nn\\
&= \underbrace{\alpha_t^2 P_t + (1-\alpha_t)^2\Sigma_{t|t-1} + 2\alpha_t(1-\alpha_t){\bf E}\left[({\bf m}_t-\hat{\bf m}_{t|t-1})({\bf x}_t-\hat{\bf x}_{t|t-1})\right]}_{\text{signal power}}\nonumber\\
&\qquad+\underbrace{\alpha_t^2\sigma^2_{{\bf v}_t} + (1-\alpha_t)^2\sigma^2_{{\bf w}_t}}_{\text{ Gaussian noise power}} \,.\label{cond_variance_r}
\end{align}
Next, we give the information theoretic characterization of the average total feedback capacity and the corresponding information feedback capacity per time instant between $\{{\bf x}_t:~t\in\mathbb{N}_0^n\}$ and $\{{\bf r}_t:~t\in\mathbb{N}_0^n\}$, denoted hereinafter by $C^{fb}_{\avertot}(\{P_t\}_{t=0}^n)$ and $C^{fb}_{t}(P_t)$, respectively.  
\begin{align}
C^{fb}_{\avertot}(\{P_t\}_{t=0}^n) &= \sup_{\substack{{\bf E}\left[({\bf m}_t-\hat{\bf m}_{t|t-1})^2\right]=P_t,~\forall{t}}}\frac{1}{n+1} I({\bf x}^n\rightarrow{\bf r}^n)\nonumber\\
&\stackrel{\mathclap{(i)}}{=}\sup_{\substack{ {\bf E}\left[({\bf m}_t-\hat{\bf m}_{t|t-1})^2\right]=P_t,~\forall{t}}}\frac{1}{n+1}\sum_{t=0}^nI({\bf x}^t;{\bf r}_t|{\bf r}^{t-1})\nonumber\\
&=\sup_{\substack{{\bf E}\left[({\bf m}_t-\hat{\bf m}_{t|t-1})^2\right]=P_t,~\forall{t}}}\frac{1}{n+1}\sum_{t=0}^n\left[h({\bf r}_t|{\bf r}^{t-1})-h({\bf r}_t|{\bf r}^{t-1},{\bf x}^t)\right]\nonumber\\
&\stackrel{\mathclap{(ii)}}{=}\sup_{\substack{{\bf E}\left[({\bf m}_t-\hat{\bf m}_{t|t-1})^2\right]=P_t,~\forall{t}}}\frac{1}{n+1}\sum_{t=0}^n\left[h^G({\bf r}_t|{\bf r}^{t-1})-h^G({\bf r}_t|{\bf r}^{t-1},{\bf x}^t)\right]\nonumber\\
&\stackrel{\mathclap{(iii)}}{=}\frac{1}{n+1}\sum_{t=0}^nC^{fb}_{t}(P_t),
\label{information_feedback_capacity_total_average}
\end{align}
where
\begin{align}
C_t^{fb}(P_t)=\frac{1}{2}\log\left(1+{\alpha_t^2 P_t + (1-\alpha_t)^2\Sigma_{t|t-1} + 2\alpha_t(1-\alpha_t){\bf E}\left[({\bf m}_t-\hat{\bf m}_{t|t-1})({\bf x}_t-\hat{\bf x}_{t|t-1})\right] \over \alpha_t^2\sigma^2_{{\bf v}_t} + (1-\alpha_t)^2\sigma^2_{{\bf w}_t}}\right), \,~t\in\mathbb{N}_0^n, \label{information_feedback_capacity_per_instant}
\end{align}
$h(\cdot|\cdot)<\infty$ is the conditional differential entropy that is assumed to be finite, $(i)$ follows by definition of directed information \cite{massey:1990}; $(ii)$ follows because the noise is additive Gaussian; $(iii)$ follows because $h^G({\bf r}_t|{\bf r}^{t-1})$ can be computed from \eqref{cond_variance_r} and $h^G({\bf r}_t|{\bf r}^{t-1},{\bf x}^t)=\frac{1}{2}\log(2\pi{e})\left(\alpha_t^2\sigma^2_{{\bf v}_t} + (1-\alpha_t)^2\sigma^2_{{\bf w}_t}\right)$ for each time instant.

Next, we describe an interesting structural result of both $C^{fb}_{\avertot}(\{P_t\}_{t=0}^n)$ and $C_t^{fb}(P_t)$.
\begin{proposition}\label{prop:structural_res_fb_cap}(Structural result) Define the following information characterization of the information feedback capacity
	\begin{align}
	\bar{C}^{fb}_{\avertot}(\{P_t\}_{t=0}^n)=\sup_{\substack{ {\bf E}\left[({\bf m}_t-\hat{\bf m}_{t|t-1})^2\right]=P_t,~\forall{t}}}\frac{1}{n+1}\sum_{t=0}^nI({\bf x}_t;{\bf r}_t|{\bf r}^{t-1}).\label{structural_fb_cap}
	\end{align} 
	Then, for the same $\{{\bf x}_t:~t\in\mathbb{N}_0^n\}$ and $\{{\bf r}_t:~t\in\mathbb{N}_0^n\}$ used to obtain $C^{fb}_{\avertot}(\{P_t\}_{t=0}^n)$ and $C_t^{fb}(P_t)$, we have that $\bar{C}^{fb}_{\avertot}(\{P_t\}_{t=0}^n)=C^{fb}_{\avertot}(\{P_t\}_{t=0}^n)=\frac{1}{n+1}\sum_{t=0}^nC^{fb}_{t}(P_t)$ where $C_t^{fb}(P_t)$=\eqref{information_feedback_capacity_per_instant} for any $t$.
\end{proposition}
\begin{proof}
	This follows by computing $\bar{C}^{fb}_{[0,n]}$ at each instant of time.
\end{proof}

Next, we derive the lower bound on the average total estimation error. Before doing it,  we first consider a lower bound on the estimation error at each time instant obtained forward in time.  To do it, we consider the following inequality:
\begin{align}
I({\bf x}^n\rightarrow{\bf r}^n)=\sum_{t=0}^nI({\bf x}^t;{\bf r}_t|{\bf r}^{t-1})\stackrel{(\ast)}\geq\sum_{t=0}^nI({\bf x}_t;{\bf r}_t|{\bf r}^{t-1})\label{useful_structural_inequality}
\end{align}
where $(\ast)$ follows by definition of directed information.
Observe that per time instant, the following series of inequalities hold: 
\begin{align}
I({\bf x}_t;{\bf r}_t|{\bf r}^{t-1})&=h({\bf x}_t|{\bf r}^{t-1})-h({\bf x}_t|{\bf r}^{t})=h({\bf x}_t|{\bf r}^{t-1})-h({\bf x}_t-{\bf E}[{\bf x}_t|{\bf r}^{t}]|{\bf r}^t)\nonumber\\
&\stackrel{\mathclap{(\star)}}\geq{h}({\bf x}_t|{\bf r}^{t-1})-h({\bf x}_t-{\bf E}[{\bf x}_t|{\bf r}^{t}])\nonumber\\
&\stackrel{\mathclap{(\star\star)}}=\frac{1}{2}\log{2\pi{e}}\Sigma_{t|t-1}-\frac{1}{2}\log{2\pi{e}}J_t^d=\frac{1}{2}\log\left(\frac{\Sigma_{t|t-1}}{J_t^d}\right),~\Sigma_{0|-1}=\sigma^2_{{\bf x}_0},~\forall{t},\nonumber\\
\Longrightarrow&{J}_t^d\geq\Sigma_{t|t-1}2^{-2I({\bf x}_t;{\bf r}_t|{\bf r}^{t-1})}\stackrel{(\star\star\star)}\geq\Sigma_{t|t-1}2^{-2C_t^{fb}}\nonumber\\
&\stackrel{\mathclap{(\star\star\star\star)}}{=}\;\;\Sigma_{t|t-1}2^{-2\frac{1}{2}\log\left(1+{\alpha_t^2 P_t + (1-\alpha_t)^2\Sigma_{t|t-1} + 2\alpha_t(1-\alpha_t){\bf E}\left[({\bf m}_t-\hat{\bf m}_{t|t-1})({\bf x}_t-\hat{\bf x}_{t|t-1})\right] \over \alpha_t^2\sigma^2_{{\bf v}_t} + (1-\alpha_t)^2\sigma^2_{{\bf w}_t}}\right)}\nonumber\\
&=\frac{\Sigma_{t|t-1}}{1+{\alpha_t^2 P_t + (1-\alpha_t)^2\Sigma_{t|t-1} + 2\alpha_t(1-\alpha_t){\bf E}\left[({\bf m}_t-\hat{\bf m}_{t|t-1})({\bf x}_t-\hat{\bf x}_{t|t-1})\right] \over \alpha_t^2\sigma^2_{{\bf v}_t} + (1-\alpha_t)^2\sigma^2_{{\bf w}_t}}}\stackrel{(\star\star\star\star\star)}\geq\mbox{eq. \eqref{optimal_estimation_error}},~\mbox{for any $t$}\label{inequalities_series}
\end{align}
where $(\star)$ follows because conditioning reduces entropy; $(\star\star)$ follows because the source process is Gauss-Markov driven by additive Gaussian noise whereas $h({\bf x}_t-\hat{\bf x}_{t|t})$ is maximized if and only if $h({\bf x}_t-\hat{\bf x}_{t|t})={h}^G({\bf x}_t-\hat{\bf x}_{t|t})$; $(\star\star\star)$ follows because $I({\bf x}_t;{\bf r}_t|{\bf r}^{t-1})\leq\sup_{{\bf E}[({\bf m}_t-\hat{\bf m}_{t|t-1})^2]={P}_t}I({\bf x}_t;{\bf r}_t|{\bf r}^{t-1})$ for any $t$; $(\star\star\star\star)$ follows from Proposition \eqref{prop:structural_res_fb_cap} and \eqref{information_feedback_capacity_per_instant}; $(\star\star\star\star\star)$ is obtained using the following series of inequalities:
\begin{align*}
&1+{\alpha_t^2 P_t + (1-\alpha_t)^2\Sigma_{t|t-1} + 2\alpha_t(1-\alpha_t){\bf E}\left[({\bf m}_t-\hat{\bf m}_{t|t-1})({\bf x}_t-\hat{\bf x}_{t|t-1})\right] \over \alpha_t^2\sigma^2_{{\bf v}_t} + (1-\alpha_t)^2\sigma^2_{{\bf w}_t}}\nonumber\\
&\overset{(p1)}{\leq} 1+{\alpha_t^2 P_t + (1-\alpha_t)^2\Sigma_{t|t-1} + 2\alpha_t(1-\alpha_t)\sqrt{P_t\Sigma_{t|t-1}} \over \alpha_t^2\sigma^2_{{\bf v}_t} + (1-\alpha_t)^2\sigma^2_{{\bf w}_t}} \\
&\overset{(p2)}{\leq} 1+{\alpha_t P_t\over\alpha_t^2\sigma_{{\bf v}_t}^2}+{(1-\alpha_t)^2\Sigma_{t|t-1}\over(1-\alpha_t)^2\sigma_{{\bf w}_t}^2} \,,
\end{align*} 
where $(p1)$ holds due to the Cauchy-Schwarz inequality;  $(p2)$ holds because of the inequality in the derivation of Theorem~\ref{thm:linearOpt} , i.e., ${\eta_t \over \upsilon_t} + {\theta_t \over \phi_t}\geq{(\sqrt{\eta_t}+\sqrt{\theta_t})^2 \over \upsilon_t+\phi_t}$ for positive $\eta_t,~\theta_t,\upsilon_t,~\phi_t$ with $\eta_t=\alpha_t^2P_t$, $\theta_t=(1-\alpha_t)^2\Sigma_{t|t-1}$, $\upsilon_t=\alpha_t^2\sigma_{{\bf v}_t}^2$, and $\phi_t=(1-\alpha_t)^2\sigma^2_{{\bf w}_t}$.
	
In \eqref{inequalities_series}, the first inequality holds with equality if and only if $({\bf x}^n,{\bf r}^n)$ are jointly Gaussian which is the case when the encoder is linear with noiseless feedback; $(\star\star\star\star\star)$ holds with equality for $0<\alpha_t<1$ when $\sqrt{\eta_t}\phi_t=\sqrt{\theta_t}\upsilon_t \Rightarrow \alpha_t\sqrt{P_t}(1-\alpha_t)^2\sigma_{{\bf w}_t}^2=(1-\alpha_t)\Sigma_{t|t-1}\alpha_t^2\sigma_{{\bf v}_t}^2 \Rightarrow \sqrt{P_t}={\alpha_t\over{1}-\alpha_t}\Sigma_{t|t-1}{\sigma_{{\bf v}_t}^2\over\sigma_{{\bf w}_t}^2}$ for any $t$. Since from Proposition \ref{prop:1} we showed that $\alpha_t^*={A_t^*\sigma_{{\bf w}_t}^2 \over A_t^*\sigma_{{\bf w}_t}^2+\sigma_{{\bf v}_t}^2}$ (from Table \ref{tab:optDecoderMulti}) for an innovations encoder $\gamma_t^e({\bf x}^t,{\bf r}^{t-1})=A_t({\bf x}_t-\hat{\bf x}_{t|t-1})$, we obtain ${\bf E}\left[({\bf m}_t-\hat{\bf m}_{t|t-1})^2\right]=P_t=A_t^2\Sigma_{t|t-1}$, which is consistent with a linear encoder with a noiseless feedback (innovations encoder). Note that for $\alpha_t=0$, inequality $(\star\star\star\star\star)$ in \eqref{inequalities_series} reduces to $\frac{\Sigma_{t|t-1}}{{\Sigma_{t|t-1}\over\sigma_{{\bf w}_t}^2}+1}$ that also holds with equality, and for $\alpha_t=1$, $(\star\star\star\star\star)$ in \eqref{inequalities_series} reduces to $\frac{\Sigma_{t|t-1}}{{P_t\over\sigma_{{\bf v}_t}^2}+1}$ that also holds with equality. Thus, the information theoretic lower bound on the estimation error at each instant of time is given by \eqref{optimal_estimation_error} and it is achievable at each instance of time only for jointly linear encoder and decoder with $A_t>0$ and $0<\alpha_t<1$.

Thus, we have proved that at each instant of time going forward in time, the information theoretic lower bound on the estimation error is achievable only for jointly linear encoder and decoder.

 The final result is obtained once we take the average total value of the estimation error at each instant of time. This completes the derivation. \qed

\section{Proof of Theorem~\ref{theorem:enc_sol_soft}} \label{appendixenc_sol_soft}
We observe that the optimization variables of interest in \eqref{enc_optimization_soft} are $\{A_t^2:~t\in\mathbb{N}_0^n\}$ hence we can introduce the decision variables $\{\mu_t=A_t^2:~t\in\mathbb{N}_0^n\}$ which are non-negative variables. Hence, \eqref{enc_optimization_soft} can be cast as follows:
\begin{align}
J_{\avertot}^{e}=\min\limits_{\mu_t\geq{0},~t\in\mathbb{N}_0^n}\frac{1}{n+1}\sum_{t=0}^n\left[{\Sigma^*_{t|t-1}\sigma^2_{{\bf w}_t}\sigma^2_{{\bf v}_t}\over (\mu_t\sigma^2_{{\bf w}_t}+\sigma^2_{{\bf v}_t})\Sigma^*_{t|t-1}+\sigma^2_{{\bf w}_t}\sigma^2_{{\bf v}_t}}+\theta_t\mu_t\Sigma_{t|t-1}+b_t^2\right],~\Sigma^*_{0|-1}=\sigma^2_{{\bf x}_0}.\label{enc_optimization_soft_alternative}
\end{align}
To solve \eqref{enc_optimization_soft_alternative}, we employ again Lagrange multipliers and forward induction. First, we write the augmented Lagrangian problem as follows
\begin{align}
{\cal L}_{\avertot}^e(\{f_t\}_{t=0}^n,~\{\mu_t\}_{t=0}^n)=&\frac{1}{n+1}\sum_{t=0}^n\Big[{\Sigma^*_{t|t-1}\sigma^2_{{\bf w}_t}\sigma^2_{{\bf v}_t}\over (\mu_t\sigma^2_{{\bf w}_t}+\sigma^2_{{\bf v}_t})\Sigma^*_{t|t-1}+\sigma^2_{{\bf w}_t}\sigma^2_{{\bf v}_t}}+\theta_t\mu_t\Sigma_{t|t-1}+b_t^2-f_t\mu_t\Big],~\Sigma^*_{0|-1}=\sigma^2_{{\bf x}_0}.\label{aug_lagrangian_soft}
\end{align}
The first order derivative test, the complementary slackness and the primal and dual feasibility conditions, respectively, are derived as follows:
\begin{align}
&{\partial{\cal L}_{\avertot}^e(\{f_t\}_{t=0}^n,~\{\mu_t\}_{t=0}^n)\over\partial\mu_t}\Bigg{|}_{\substack{\mu_t=\mu_t^*\\f_t=f_t^{*}}}=0,~t=0,1,\ldots,n\label{stationarity_1_soft}\\
&f_t\mu_t=0,~\forall{t},\label{complementary_slackness_soft}\\
&\mu_t\geq{0},~\forall{t},\label{primal_feasibility_soft}\\
&f_t\geq{0},~\forall{t}.\label{dual_feasibility_soft}
\end{align}
Next, we optimize forward in time ${\cal L}^e_{\avertot}(\cdot)$ and study every possible scenario depending of the active variables.\\
\underline{t=0:}
\begin{align}
\begin{split}
&{\partial{\cal L}_{\avertot}^e(\{\mu_t\}_{t=0}^n)\over\partial\mu_0}\Bigg{|}_{\substack{\mu_0=\mu_0^*\\f_0=f_0^{*}}}=0\\
\Longrightarrow&\left[-\frac{1}{\sigma^2_{{\bf v}_0}\left(\frac{\mu^*_0}{\sigma^2_{{\bf v}_0}}+\frac{1}{\sigma^2_{{\bf w}_0}}+\frac{1}{\sigma^2_{{\bf x}_0}}\right)^2}\right]+\theta_0\sigma^2_{{\bf x}_0}-f_0^{*}=0\Longrightarrow\theta_0=\frac{f_0^{*}}{\sigma^2_{{\bf x}_0}}+\left[\frac{1}{\sigma^2_{{\bf v}_0}\sigma^2_{{\bf x}_0}\left(\frac{\mu^*_0}{\sigma^2_{{\bf v}_0}}+\frac{1}{\sigma^2_{{\bf w}_0}}+\frac{1}{\sigma^2_{{\bf x}_0}}\right)^2}\right].
\end{split}\label{time_stage_0_soft}
\end{align}
Next, we check possible cases to obtain our results when ${\theta_t}>0$ is given.\\
{\bf Case 1:} Let $\mu_0^*=0$. Then from \eqref{complementary_slackness_soft}, $f_0^{*}\geq{0}$, which in turn implies from \eqref{time_stage_0_soft} that
\begin{align}
\theta_0=\frac{f_0^{*}}{\sigma^2_{{\bf x}_0}}+\left[\frac{1}{\sigma^2_{{\bf v}_0}\sigma^2_{{\bf x}_0}\left(\frac{1}{\sigma^2_{{\bf w}_0}}+\frac{1}{\sigma^2_{{\bf x}_0}}\right)^2}\right]\geq\frac{1}{\sigma^2_{{\bf v}_0}\sigma^2_{{\bf x}_0}\left(\frac{1}{\sigma^2_{{\bf w}_0}}+\frac{1}{\sigma^2_{{\bf x}_0}}\right)^2}\equiv\theta_0^{\prime}.\label{time_Stage_0_soft_case_1}
\end{align}
{\bf Case 2:} Now assume that $\mu_0^*>0$. Then, from \eqref{complementary_slackness_soft} we obtain that $f_0^{*}=0$, which implies from \eqref{time_stage_0_soft} that
\begin{align}
\theta_0=\frac{1}{\sigma^2_{{\bf v}_0}\sigma^2_{{\bf x}_0}\left(\frac{\mu^*_0}{\sigma^2_{{\bf v}_0}}+\frac{1}{\sigma^2_{{\bf w}_0}}+\frac{1}{\sigma^2_{{\bf x}_0}}\right)^2}<\theta_0^{\prime}.\label{time_Stage_0_soft_case_2}
\end{align}
Moreover, solving in \eqref{time_Stage_0_soft_case_2} the equation w.r.t. $\mu_0^*$ we obtain
\begin{align}
\mu_0^*=\frac{\sigma^2_{{\bf v}_0}}{\sqrt{\theta_0\sigma^2_{{\bf x}_0}\sigma^2_{{\bf v}_0}}}-\frac{\sigma^2_{{\bf v}_0}}{\sigma^2_{{\bf x}_0}}\left(\frac{\sigma^2_{{\bf x}_0}}{\sigma^2_{{\bf w}_0}}+1\right).\label{lagrangian_time_Stage_0_case_2}
\end{align}
Clearly, from the first order derivative in \eqref{time_stage_0_soft}, we can easily see that the second derivative w.r.t. to $\mu^*_0$ is positive hence the function is convex and the optimal solution at this stage is global. \\
\underline{t=1:}
\begin{align}
\begin{split}
&{\partial{\cal L}_{\avertot}^e(\mu_{0}^*,~\{\mu_t\}_{t=1}^n)\over\partial\mu_1}\Bigg{|}_{\substack{\mu_1=\mu_1^*\\f_1=f_1^{*}}}=0\\
\Longrightarrow&\left[-\frac{1}{\sigma^2_{{\bf v}_1}\left(\frac{\mu^*_1}{\sigma^2_{{\bf v}_1}}+\frac{1}{\sigma^2_{{\bf w}_1}}+\frac{1}{\Sigma^*_{1|0}}\right)^2}\right]+\theta_1\Sigma^*_{1|0}-f_1^{*}=0\Longrightarrow\theta_1=\frac{f_1^{*}}{\Sigma^*_{1|0}}+\left[\frac{1}{\sigma^2_{{\bf v}_0}\Sigma^*_{1|0}\left(\frac{\mu^*_1}{\sigma^2_{{\bf v}_1}}+\frac{1}{\sigma^2_{{\bf w}_1}}+\frac{1}{\Sigma^*_{1|0}}\right)^2}\right].
\end{split}\label{time_stage_1_soft}
\end{align}
At this stage we note that $\Sigma^*_{1|0}$ is independent of $\mu_1^*$ because its optimal solution depends on $\mu_0^*$ that is already obtained at time stage $0$.  Hence, under this observation, we can follow precisely the approached followed in time stage zero which will give 
Next, we check possible cases to obtain our results.\\
{\bf Case 1:} Let $\mu_1^*=0$. Then from \eqref{complementary_slackness_soft}, $f_1^{*}\geq{0}$, which in turn implies from \eqref{time_stage_1_soft} that
\begin{align}
\theta_1=\frac{f_1^{*}}{\Sigma^*_{1|0}}+\left[\frac{1}{\sigma^2_{{\bf v}_1}\Sigma^*_{1|0}\left(\frac{1}{\sigma^2_{{\bf w}_1}}+\frac{1}{\Sigma^*_{1|0}}\right)^2}\right]\geq\frac{1}{\sigma^2_{{\bf v}_1}\Sigma^*_{1|0}\left(\frac{1}{\sigma^2_{{\bf w}_1}}+\frac{1}{\Sigma^*_{1|0}}\right)^2}\equiv\theta_t^\prime.\label{time_Stage_1_soft_case_1}
\end{align}
{\bf Case 2:} Now assume that $\mu_1^*>0$. Then, from \eqref{complementary_slackness_soft} we obtain that $f_1^{*}=0$, which implies from \eqref{time_stage_1_soft} that
\begin{align}
\theta_1=\frac{1}{\sigma^2_{{\bf v}_1}\Sigma^*_{1|0}\left(\frac{\mu^*_1}{\sigma^2_{{\bf v}_1}}+\frac{1}{\sigma^2_{{\bf w}_1}}+\frac{1}{\Sigma^*_{1|0}}\right)^2}<\theta_1^\prime.\label{time_Stage_1_soft_case_2}
\end{align}
Moreover, solving the equality in \eqref{time_Stage_1_soft_case_2} w.r.t. $\mu_1^*$ we obtain
\begin{align}
\mu_1^*=\frac{\sigma^2_{{\bf v}_1}}{\sqrt{\theta_1\Sigma^*_{1|0}\sigma^2_{{\bf v}_1}}}-\frac{\sigma^2_{{\bf v}_1}}{\Sigma^*_{1|0}}\left(\frac{\Sigma^*_{1|0}}{\sigma^2_{{\bf w}_1}}+1\right).\label{lagrangian_time_Stage_1_case_2}
\end{align}
Clearly, from the first order derivative in \eqref{time_stage_1_soft}, we can easily see that the second derivative w.r.t. to $\mu^*_1$ is positive hence the function is convex and the optimal solution at this stage is global. \\
Now suppose that at time $n-1$ the optimal solution of $\mu^*_{n-1}$, for the possible cases is as follows:\\
{\bf Case 1:} Let $\mu_{n-1}^*=0$. Then from \eqref{complementary_slackness_soft}, $f_{n-1}^{*}\geq{0}$, which in turn implies that
\begin{align}
\theta_{n-1}=\frac{f_{n-1}^{*}}{\Sigma^*_{n-1|n-2}}+\left[\frac{1}{\sigma^2_{{\bf v}_{n-1}}\Sigma^*_{n-1|n-2}\left(\frac{1}{\sigma^2_{{\bf w}_{n-1}}}+\frac{1}{\Sigma^*_{n-1|n-2}}\right)^2}\right]\geq\frac{1}{\sigma^2_{{\bf v}_{n-1}}\Sigma^*_{n-1|n-2}\left(\frac{1}{\sigma^2_{{\bf w}_{n-1}}}+\frac{1}{\Sigma^*_{n-1|n-2}}\right)^2}\equiv\theta_{n-1}^{\prime}.\label{time_Stage_n-1_soft_case_1}
\end{align}
{\bf Case 2:} Now assume that $\mu_{n-1}^*>0$. Then, from \eqref{complementary_slackness_soft} we obtain that $f_{n-1}^{*}=0$, which implies that
\begin{align}
\theta_{n-1}=\frac{1}{\sigma^2_{{\bf v}_{n-1}}\Sigma^*_{n-1|n-2}\left(\frac{\mu^*_{n-1}}{\sigma^2_{{\bf v}_{n-1}}}+\frac{1}{\sigma^2_{{\bf w}_{n-1}}}+\frac{1}{\Sigma^*_{n-1|n-2}}\right)^2}<\theta_{n-1}^{\prime}.\label{time_Stage_n-1_soft_case_2}
\end{align}
Moreover, solving the equation in\eqref{time_Stage_n-1_soft_case_2} w.r.t. $\mu_{n-1}^*$ we obtain
\begin{align}
\mu_{n-1}^*=\frac{\sigma^2_{{\bf v}_{n-1}}}{\sqrt{\theta_{n-1}\Sigma^*_{n-1|n-2}\sigma^2_{{\bf v}_{n-1}}}}-\frac{\sigma^2_{{\bf v}_{n-1}}}{\Sigma^*_{n-1|n-2}}\left(\frac{\Sigma^*_{n-1|n-2}}{\sigma^2_{{\bf w}_{n-1}}}+1\right).\label{lagrangian_time_Stage_n-1_case_2}
\end{align}
Then, at time stage $t=n$, we can obtain following the same argument as in time $t=1$ that the followin cases hold. \\
{\bf Case 1:} Let $\mu_{n}^*=0$. Then from \eqref{complementary_slackness_soft}, $f_{n}^{*}\geq{0}$, which in turn implies that
\begin{align}
\theta_{n}=\frac{f_{n}^{*}}{\Sigma^*_{n|n-1}}+\left[\frac{1}{\sigma^2_{{\bf v}_{n}}\Sigma^*_{n|n-1}\left(\frac{1}{\sigma^2_{{\bf w}_{n}}}+\frac{1}{\Sigma^*_{n|n-1}}\right)^2}\right]\geq\frac{1}{\sigma^2_{{\bf v}_{n}}\Sigma^*_{n|n-1}\left(\frac{1}{\sigma^2_{{\bf w}_{n}}}+\frac{1}{\Sigma^*_{n|n-1}}\right)^2}\equiv\theta^{\prime}_n.\label{time_Stage_n_soft_case_1}
\end{align}
{\bf Case 2:} Now assume that $\mu_{n}^*>0$. Then, from \eqref{complementary_slackness_soft} we obtain that $f_{n}^{*}=0$, which implies that
\begin{align}
\theta_{n}=\frac{1}{\sigma^2_{{\bf v}_{n}}\Sigma^*_{n|n-1}\left(\frac{\mu^*_{n}}{\sigma^2_{{\bf v}_{n}}}+\frac{1}{\sigma^2_{{\bf w}_{n}}}+\frac{1}{\Sigma^*_{n|n-1}}\right)^2}<\theta_{n}^{\prime}.\label{time_Stage_n_soft_case_2}
\end{align}
Moreover, solving \eqref{time_Stage_n_soft_case_2} w.r.t. $\mu_{n}^*$ we obtain
\begin{align}
\mu_{n}^*=\frac{\sigma^2_{{\bf v}_n}}{\sqrt{\theta_n\Sigma^*_{n|n-1}\sigma^2_{{\bf v}_n}}}-\frac{\sigma^2_{{\bf v}_n}}{\Sigma^*_{n|n-1}}\left(\frac{\Sigma^*_{n|n-1}}{\sigma^2_{{\bf w}_n}}+1\right).\label{lagrangian_time_Stage_n_case_2}
\end{align}
Hence, we proved that by optimizing forward in time, we obtain the optimal $\{\mu_t^*:~t\in\mathbb{N}_0^n\}$. The problem is solved once we replace $\mu_t^*=A_t^{2,*}$,~for $t=0,1,\ldots,n$, in \eqref{enc_optimization_soft} which leads to \eqref{enc_optimization_sol_soft}, \eqref{individual_cost_soft}, \eqref{individual_prime_cost}, \eqref{cases_optimal_A_soft} and \eqref{covariance_cost_soft}. This completes the derivation. \qed

\section{Proof of Theorem~\ref{thm:affineNash}} \label{appendixaffineNash} 
	\begin{itemize}
	\item[(i)] For the given affine encoder strategy ${\bf m}=\gamma^e({\bf x})=A{\bf x}+C$, the decoder input is
	\[ {\bf r} = (\alpha A+1-\alpha){\bf x} + \alpha {\bf v} + (1-\alpha){\bf w} + \alpha C\]
	when the decoder adjusts the time-sharing parameter $\alpha$ of the channels. Then, similar to Theorem~\ref{thm:affineDecoder}, the optimal decoder strategy is $\gamma^d({\bf r})=\hat{{\bf x}}={\bf E}[{\bf x}|{\bf r}]$. For $A>0$, it can be expressed as
	\begin{align}
	\gamma^d({\bf r}) = {A\sigma_{\bf x}^2\sigma_{\bf w}^2 + \sigma_{\bf x}^2\sigma_{\bf v}^2\over A^2\sigma_{\bf x}^2\sigma_{\bf w}^2+\sigma_{\bf x}^2\sigma_{\bf v}^2+\sigma_{\bf w}^2\sigma_{\bf v}^2}({\bf r}-\alpha C) 
	\label{eq:optDecoderEstNash}
	\end{align}
	with the channel combining parameter $\alpha = {A\sigma_{\bf w}^2 \over A\sigma_{\bf w}^2+\sigma_{\bf v}^2}$.\newline
	For $-\sqrt{\sigma_{\bf v}^2\over\sigma_{\bf w}^2}\leq A\leq0$, we have $\gamma^d({\bf r})={\sigma_{\bf x}^2 \over \sigma_{\bf x}^2+\sigma_{\bf w}^2}{\bf r}$ and $\alpha=0$. \newline
	For $A\leq-\sqrt{\sigma_{\bf v}^2\over\sigma_{\bf w}^2}$, we have $\gamma^d({\bf r})={A\sigma_{\bf x}^2 \over A^2\sigma_{\bf x}^2+\sigma_{\bf v}^2}({\bf r}- C)$ and $\alpha=1$.
	\item[(ii)] For the given affine decoder strategy $\hat{{\bf x}}=\gamma^d({\bf r})=K{\bf r}+L$ and the nonzero channel combining parameter $\alpha$, since  ${\bf r} = \alpha(\gamma^e({\bf x})+{\bf v}) + (1-\alpha)({\bf x}+{\bf w})$, we have $\hat{{\bf x}}=\alpha K \gamma^e({\bf x}) + (1-\alpha)K{\bf x} + \alpha K{\bf v} + (1-\alpha)K{\bf w}+L$. Then, the corresponding encoder cost is
	\begin{align*}
	J^e &= {\bf E}[({\bf x}-\hat{{\bf x}}-b)^2]+\theta{\bf E}[(\gamma^e({\bf x}))^2] \nn\\
	&= {\bf E}\big[(-\alpha K \gamma^e({\bf x})+(1-(1-\alpha)K){\bf x} -L-b)^2 + \theta(\gamma^e({\bf x}))^2\big] + \alpha^2K^2\sigma_{\bf v}^2 + (1-\alpha)^2K^2\sigma_{\bf w}^2 \nn\\
	&= {\bf E}\big[(\alpha^2K^2+\theta)(\gamma^e({\bf x}))^2 - 2 \alpha K((1-(1-\alpha)K){\bf x} -L-b)\gamma^e({\bf x}) + ((1-(1-\alpha)K){\bf x} -L-b)^2\big] \nn\\
	&\qquad\qquad+ \alpha^2K^2\sigma_{\bf v}^2 + (1-\alpha)^2K^2\sigma_{\bf w}^2 \\
	&= (\alpha^2K^2+\theta) {\bf E}\Bigg[\left(\gamma^e({\bf x})-{\alpha K((1-(1-\alpha)K){\bf x} -L-b)\over\alpha^2K^2+\theta}\right)^2 - \left({\alpha K((1-(1-\alpha)K){\bf x} -L-b)\over\alpha^2K^2+\theta}\right)^2\Bigg] \nn\\
	&\qquad\qquad+ (1-(1-\alpha)K)^2\sigma_{\bf x}^2+(L+b)^2+\alpha^2K^2\sigma_{\bf v}^2 + (1-\alpha)^2K^2\sigma_{\bf w}^2 \nn\\
	&= (\alpha^2K^2+\theta) {\bf E}\Bigg[\left(\gamma^e({\bf x})-{\alpha K((1-(1-\alpha)K){\bf x} -L-b)\over\alpha^2K^2+\theta}\right)^2\Bigg] \nn\\
	&\qquad\qquad+ \theta\left({(1-(1-\alpha)K)^2\sigma_{\bf x}^2+(L+b)^2\over\alpha^2K^2+\theta}\right) +\alpha^2K^2\sigma_{\bf v}^2 + (1-\alpha)^2K^2\sigma_{\bf w}^2\,.
	\end{align*}
	Thus, the optimal encoder strategy that minimizes the encoder cost is 
	\begin{align}
	\gamma^e({\bf x}) = {\alpha K((1-(1-\alpha)K){\bf x} -L-b)\over\alpha^2K^2+\theta} \,.
	\label{eq:optEncoderSoftNash}
	\end{align}	
	\item[(iii)] In order to have an affine Nash equilibrium, the best responses of the encoder and the decoder must match each other. In particular, for $A>0$,  \eqref{eq:optDecoderEstNash} and \eqref{eq:optEncoderSoftNash}  must be simultaneously satisfied:
	\begin{align*}
	A &= {\alpha K(1-(1-\alpha)K)\over\alpha^2K^2+\theta} \,,\quad C=-{\alpha K(L+b)\over\alpha^2K^2+\theta}\,, \quad
	K = {A\sigma_{\bf x}^2\sigma_{\bf w}^2 + \sigma_{\bf x}^2\sigma_{\bf v}^2\over A^2\sigma_{\bf x}^2\sigma_{\bf w}^2+\sigma_{\bf x}^2\sigma_{\bf v}^2+\sigma_{\bf w}^2\sigma_{\bf v}^2} \,,\quad L=-\alpha KC \,,\quad
	\alpha = {A\sigma_{\bf w}^2 \over A\sigma_{\bf w}^2+\sigma_{\bf v}^2} \,.
	\end{align*} 
	Notice the following:
	\begin{align}
	A &= {\alpha K(1-(1-\alpha)K)\over\alpha^2K^2+\theta} = 
	{{A\sigma_{\bf w}^2 \over A^2\sigma_{\bf w}^2+\sigma_{\bf v}^2+{\sigma_{\bf w}^2\sigma_{\bf v}^2\over\sigma_{\bf x}^2}}\left({1-{\sigma_{\bf v}^2 \over A^2\sigma_{\bf w}^2+\sigma_{\bf v}^2+{\sigma_{\bf w}^2\sigma_{\bf v}^2\over\sigma_{\bf x}^2}}}\right) \over \left({A\sigma_{\bf w}^2 \over A^2\sigma_{\bf w}^2+\sigma_{\bf v}^2+{\sigma_{\bf w}^2\sigma_{\bf v}^2\over\sigma_{\bf x}^2}}\right)^2+\theta} = {A\sigma_{\bf w}^2\left(A^2\sigma_{\bf w}^2+{\sigma_{\bf w}^2\sigma_{\bf v}^2\over\sigma_{\bf x}^2}\right) \over A^2\sigma_{\bf w}^4 + \theta\left({ A^2\sigma_{\bf w}^2+\sigma_{\bf v}^2+{\sigma_{\bf w}^2\sigma_{\bf v}^2\over\sigma_{\bf x}^2}}\right)^2} \nn\\ 
	&\Rightarrow \theta\left({ A^2\sigma_{\bf w}^2+\sigma_{\bf v}^2+{\sigma_{\bf w}^2\sigma_{\bf v}^2\over\sigma_{\bf x}^2}}\right)^2 = {\sigma_{\bf w}^4\sigma_{\bf v}^2\over\sigma_{\bf x}^2} \nn\\
	&\Rightarrow  \left({ A^2+{\sigma_{\bf v}^2\over\sigma_{\bf w}^2}+{\sigma_{\bf v}^2\over\sigma_{\bf x}^2}}\right)^2 = {1\over\theta}{\sigma_{\bf v}^2\over\sigma_{\bf x}^2} \nn\\
	&\Rightarrow A = \sqrt{\sqrt{{1\over\theta}{\sigma_{\bf v}^2\over\sigma_{\bf x}^2}} - {\sigma_{\bf v}^2\over\sigma_{\bf w}^2}-{\sigma_{\bf v}^2\over\sigma_{\bf x}^2}}\,. \label{eq:nashSoftEnc}
	\end{align}
	Then, by utilizing \eqref{eq:nashSoftEnc}, $K$ and $\alpha$ can be decided correspondingly. In order to have a valid encoder strategy, i.e., $A>0$, it must be satisfied that
	\begin{align*}
	\sqrt{{1\over\theta}{\sigma_{\bf v}^2\over\sigma_{\bf x}^2}} &- {\sigma_{\bf v}^2\over\sigma_{\bf w}^2}-{\sigma_{\bf v}^2\over\sigma_{\bf x}^2} > 0 \Rightarrow {1\over\theta}{\sigma_{\bf v}^2\over\sigma_{\bf x}^2} > \left({\sigma_{\bf v}^2\over\sigma_{\bf w}^2}+{\sigma_{\bf v}^2\over\sigma_{\bf x}^2}\right)^2 \Rightarrow \theta < {{\sigma_{\bf v}^2\over\sigma_{\bf x}^2} \over \left({\sigma_{\bf v}^2\over\sigma_{\bf w}^2}+{\sigma_{\bf v}^2\over\sigma_{\bf x}^2}\right)^2} = {{\sigma_{\bf x}^2\over\sigma_{\bf v}^2} \over \left({\sigma_{\bf x}^2\over\sigma_{\bf w}^2}+1\right)^2} \,.
	\end{align*}
	Thus, the linear part of the strategies (i.e., $A$ and $K$) construct consistent equations. Regarding the translation parts, observe the following:
	\begin{align*}
	L &= -\alpha CK = \alpha {\alpha K(L+b)\over\alpha^2K^2+\theta}K = {\alpha^2 K^2(L+b)\over\alpha^2K^2+\theta} \Rightarrow L \left(1-{\alpha^2K^2\over\alpha^2K^2+\theta}\right)={\alpha^2 K^2b\over\alpha^2K^2+\theta} \nn\\
	&\Rightarrow L={\alpha^2 K^2b\over\theta} \Rightarrow C=-{\alpha Kb\over\theta} \,.
	\end{align*} 
	As a result, when $\theta < {{\sigma_{\bf x}^2\over\sigma_{\bf v}^2} \over \left({\sigma_{\bf x}^2\over\sigma_{\bf w}^2}+1\right)^2}$, the jointly affine encoder and decoder strategies $\gamma^e({\bf x})=A{\bf x}+C$ and $\gamma^d({\bf r})=K{\bf r}+L$ and the channel combining parameter $\alpha$ form a Nash equilibrium.  

	Now consider the case when $A\leq-\sqrt{\sigma_{\bf v}^2\over\sigma_{\bf w}^2}$, which implies the following
	must be simultaneously satisfied:
	\begin{align*}
	A &= {\alpha K(1-(1-\alpha)K)\over\alpha^2K^2+\theta} \,,\quad C=-{\alpha K(L+b)\over\alpha^2K^2+\theta}\,,\quad
	K = {A\sigma_{\bf x}^2 \over A^2\sigma_{\bf x}^2+\sigma_{\bf v}^2} \,,\quad L=-KC \,,\quad \alpha = 1 \,.
	\end{align*} 
	Notice the following:
	\begin{align*}
	AK &= {K^2\over K^2+\theta}={A^2\sigma_{\bf x}^2 \over A^2\sigma_{\bf x}^2+\sigma_{\bf v}^2} \Rightarrow {\theta\over K^2+\theta}={\sigma_{\bf v}^2 \over A^2\sigma_{\bf x}^2+\sigma_{\bf v}^2}={\sigma_{\bf v}^2 \over {A\sigma_{\bf x}^2\over K}}={\sigma_{\bf v}^2 \over {\sigma_{\bf x}^2\over K^2+\theta}} \nn\\
	&\Rightarrow (K^2+\theta)^2 = {\theta\sigma_{\bf x}^2\over \sigma_{\bf v}^2} \Rightarrow K = \pm \sqrt{\sqrt{\theta\sigma_{\bf x}^2\over \sigma_{\bf v}^2}-\theta} \Rightarrow A = \pm \sqrt{\sqrt{\sigma_{\bf v}^2\over\theta\sigma_{\bf x}^2 }-{\sigma_{\bf v}^2\over\sigma_{\bf x}^2 }}\,.
	\end{align*}
	Note that in order to have valid strategies, it must hold that $\sqrt{\theta\sigma_{\bf x}^2\over \sigma_{\bf v}^2}-\theta >0 \Rightarrow \theta<{\sigma_{\bf x}^2\over \sigma_{\bf v}^2}$. Due to the assumption, we have $\theta<{{\sigma_{\bf x}^2\over\sigma_{\bf v}^2} \over \left({\sigma_{\bf x}^2\over\sigma_{\bf w}^2}+1\right)^2}<{\sigma_{\bf x}^2\over \sigma_{\bf v}^2}$, which satisfies the validity of strategies. Furthermore, we must also have $A\leq-\sqrt{\sigma_{\bf v}^2\over\sigma_{\bf w}^2}$, thus the negative solution of $A$ (which also implies the negative solution of $K$) will be preferred. In particular, the following must hold:
	\begin{align*}
	A &= - \sqrt{\sqrt{\sigma_{\bf v}^2\over\theta\sigma_{\bf x}^2 }-{\sigma_{\bf v}^2\over\sigma_{\bf x}^2 }} \leq-\sqrt{\sigma_{\bf v}^2\over\sigma_{\bf w}^2} \Rightarrow \sqrt{\sigma_{\bf v}^2\over\theta\sigma_{\bf x}^2 }-{\sigma_{\bf v}^2\over\sigma_{\bf x}^2} \geq {\sigma_{\bf v}^2\over\sigma_{\bf w}^2} \Rightarrow \theta \leq {{\sigma_{\bf v}^2\over\sigma_{\bf x}^2} \over {\sigma_{\bf v}^2\over\sigma_{\bf w}^2}+{\sigma_{\bf v}^2\over\sigma_{\bf x}^2}} \,,
	\end{align*}
	which is satisfied by the assumption. Thus, the linear parts of the strategies (i.e., $A$ and $K$) construct consistent equations. Regarding the translation parts, observe the following:
	\begin{align*}
	L&=- KC = {K^2(L+b)\over K^2+\theta} \Rightarrow L={ K^2b\over\theta} \Rightarrow C=-{ Kb\over\theta} \,.
	\end{align*} 

	If $-\sqrt{\sigma_{\bf v}^2\over\sigma_{\bf w}^2}\leq A\leq0$ holds, then the decoder does not utilize any information from the encoder, which implies the following
	must be simultaneously satisfied:
	\begin{align*}
	A &= {\alpha K(1-(1-\alpha)K)\over\alpha^2K^2+\theta} \,,\quad C=-{\alpha K(L+b)\over\alpha^2K^2+\theta}\,,\quad
	K = {\sigma_{\bf x}^2 \over \sigma_{\bf x}^2+\sigma_{\bf w}^2} \,,\quad L=0 \,,\quad \alpha = 0 \,.
	\end{align*} 
	Thus, $A=C=0$ is obtained. Note that, in this particular case, since the encoder has no effect on the estimation performance of the decoder, the encoder prefers not to transmit any message to minimize its cost (by avoiding transmission cost).
\end{itemize}
This completes the derivation. \qed


%



\bibliographystyle{IEEEtran}
\bibliography{isit2021Bibliography}

%
%
%
%
%
%
%

\end{document}